\documentclass{amsart}
\usepackage{amscd,amsmath,amssymb,amsfonts,bbm}
\usepackage[T1]{fontenc}
\usepackage{lmodern}

\usepackage{mathtools}
\usepackage{enumerate}
\usepackage{multicol}
\usepackage[all]{xy}

\usepackage{footmisc}

\newtheorem{thm}{Theorem}[section]
\newtheorem{prop}[thm]{Proposition}
\newtheorem{lem}[thm]{Lemma}
\newtheorem{cor}[thm]{Corollary}

\theoremstyle{definition}

\newtheorem{quest}[thm]{Question}

\theoremstyle{remark}

\numberwithin{equation}{section}

\DeclareMathOperator{\Br}{Br}

 \DeclareMathOperator{\Spec}{Spec}

\DeclareMathOperator{\Ker}{Ker}

\DeclareMathOperator{\Ima}{Im}

\DeclareMathOperator{\R}{\mathbf{R}}
\DeclareMathOperator{\C}{\mathbf{C}}
\DeclareMathOperator{\s}{\mathbf{S}}

\DeclareMathOperator{\Gal}{Gal}
\DeclareMathOperator{\Z}{\mathbb{Z}}
\DeclareMathOperator{\nr}{nr}
\DeclareMathOperator{\CH}{CH}

\begin{document}

\title[On Hilbert's 17\textsuperscript{th} problem in low degree]
{On Hilbert's 17\textsuperscript{th} problem in low degree}
\author{Olivier Benoist}
\address{Institut de Recherche Math\'ematique Avanc\'ee\\
UMR 7501, Universit\'e de Strasbourg et CNRS\\
7 rue Ren\'e Descartes\\
67000 Strasbourg, FRANCE}
\email{olivier.benoist@unistra.fr}

\renewcommand{\abstractname}{Abstract}
\begin{abstract}

 Artin solved Hilbert's 17\textsuperscript{th} problem, proving that a real polynomial in $n$ variables that is positive semidefinite is a sum of squares of rational functions, and Pfister showed that only $2^n$ squares are needed. 


 In this paper, we investigate situations where Pfister's theorem may be improved. We show that a real polynomial of degree $d$ in $n$ variables that is positive semidefinite is a sum of $2^n-1$ squares of rational functions if $d\leq 2n-2$. If $n$ is even or equal to $3$ or $5$, this result also holds for $d=2n$.
\end{abstract}
\maketitle


\section*{Introduction}\label{intro}

\subsection{Hilbert's 17\textsuperscript{th} problem}

Let $\R$ be a real closed field, for instance the field $\mathbb{R}$ of real numbers, and let $n\geq 1$. A polynomial $f\in \R[X_1,\dots,X_n]$ is said to be positive semidefinite, if $f(x_1,\dots, x_n)\geq 0$ for all $x_1,\dots,x_n\in \R$.  As an odd degree polynomial changes sign, such a polynomial has even degree.

In \cite{Artin17}, Artin answered Hilbert's 17\textsuperscript{th} problem by proving that a positive semidefinite polynomial $f\in \R[X_1,\dots,X_n]$ is a sum of squares of rational functions\footnote{Hilbert himself \cite{Hilbert} had given examples of positive semidefinite polynomials that are not sums of squares of polynomials.}.
This theorem was later improved by Pfister \cite[Theorem 1]{Pfister17} who showed that it is actually the sum of $2^n$ squares of rational functions.
We refer to \cite[Chapter 6]{Pfisterbook} for a nice account of these classical results.

In two variables, the situation is very well understood. Hilbert \cite{Hilbert} has shown that 
a positive semidefinite polynomial $f\in \R[X_1,X_2]$ of degree $\leq 4$ is a sum of $3$ squares of rational functions\footnote{In fact, in this  exceptional case, Hilbert actually showed that squares of polynomials suffice. We will not consider this question in what follows, and refer the interested reader to \cite{PfiSch}.}, and Cassels, Ellison and Pfister \cite{CEP} have given an example of a positive semidefinite polynomial $f\in \R[X_1,X_2]$ of degree $6$ that is not a sum of $3$ squares of rational functions.

Our goal is to prove an analogue of Hilbert's result -- that in low degree, less squares are needed -- in more than two variables:

\begin{thm}\label{main}
Let $f\in \R[X_1,\dots,X_n]$ be a positive semidefinite polynomial of degree $d$. Suppose that one of the following holds:
\begin{enumerate}[(i)]
\item $d\leq 2n-2$.
\item $d=2n$, and either $n$ is even, or $n=3$, or $n=5$. 
\end{enumerate}
Then $f$ is a sum of $2^n-1$ squares in $\R(X_1,\dots,X_n)$.
\end{thm}

Of course, when $d=2$, the classification of quadratic forms over $\R$ shows the much stronger result that $n+1$ squares are enough. However, to the best of our knowledge, our theorem is already new for $d=4$ and $n\geq 3$.

\subsection{Dependence on the degree}

 The question whether the bound $2^n$ in Pfister's aforementioned theorem is optimal is natural and well-known \cite[\S 4 Problem 1]{PfisterICM}. It is often formulated in the following equivalent way, where the Pythagoras number $p(K)$ of a field $K$ is the smallest number $p$ such that every sum of squares in $K$ is a sum of $p$ squares:

\begin{quest}\label{qPyt}
Do we have $p(\R(X_1,\dots,X_n))=2^n$ ?
\end{quest}

 When $n\geq 2$, the best known result is that $n+2\leq p(\R(X_1,\dots,X_n))\leq 2^n$ \cite[p. 97]{Pfisterbook}, where the upper bound is Pfister's theorem and the lower bound is an easy consequence of the Cassels-Ellison-Pfister theorem.

Our main theorem does not address this question directly: it explores the opposite direction, that is the values of the degree for which Pfister's bound may be improved.
However, Theorem \ref{main} gives insights into Question \ref{qPyt}. The bound $d\leq 2n$ has a natural geometric origin (it reflects the rational connectedness of an associated algebraic variety), and it would be natural to expect that Theorem \ref{main} cannot be extended to degrees $d\geq 2n+2$. 

In view of Theorem \ref{main}, it is natural to ask whether the bound $d\leq 2n-2$ may be improved to $d\leq 2n$ for every odd value of $n$. When $n=1$, this is not the case because $X_1^2+1$ is not a square. On the other hand, when $n\geq 3$ is odd, we reduce this question to a  geometric coniveau estimate (Proposition \ref{conditional}). When $n=3$, it is very easy to check. We also verify it when $n=5$, following an argument of Voisin. This explains the hypotheses on the degree in Theorem \ref{main}.


\subsection{Strategy of the proof}

In two variables, the theorems of Hilbert and Cassels-Ellison-Pfister quoted above have received geometric proofs by Colliot-Th\'el\`ene in \cite[Remark 2]{CTrealrat} and \cite{CTNL}. His idea is to consider the homogenization $F$ of $f$ and to introduce the algebraic surface  $Y:=\{Z^2+F=0\}$. Then, whether or not $f$ may be written as a sum of three squares in $\R(X_1,X_2)$ depends on the injectivity of the map $\Br(\R)\to \Br(\R(Y))$, which may be studied by geometrical methods.

We follow the same strategy in more variables. Proposition \ref{sumlevel} and Proposition \ref{levelnr} translate the property that $f$ is a sum of $2^n-1$ squares in $\R(X_1,\dots,X_n)$ into a cohomological property of (a resolution of singularities of) the variety $Y$. The group that plays a role analogous to that of the Brauer group in two variables is a degree $n$ unramified cohomology group. 

It remains to show that, when the degree of $f$ is small, some class in a degree $n$ unramified cohomology group vanishes.  This is more difficult than the corresponding result in two variables, as these groups are harder to control than Brauer groups. Our main tool to achieve this is Bloch-Ogus theory.

\subsection{Structure of the paper}

  The first two sections gather general cohomological results 
for varieties over $\R$, that are used throughout the text.
  It will be very important for us to use cohomology with integral coefficients (as opposed to $2$-torsion coefficients). For this reason, section \ref{s1} is devoted to general properties of the $2$-adic cohomology\footnote{It would also have been possible to work with equivariant Betti cohomology over the field $\mathbb{R}$ of real numbers \cite{Krasnov}, and with its semi-algebraic counterpart over a general real closed field.} of varieties over $\R$.
 
 In section \ref{s2}, we recall the basics of Bloch-Ogus theory, then focus on the specific properties of it over real closed fields.
In particular, we adapt to our needs a strategy of Colliot-Th\'el\`ene and Scheiderer \cite{CTS} to compare the Bloch-Ogus theory of a variety over $\R$ and over the algebraic closure $\C$ of $\R$, and explain in our context consequences of the Bloch-Kato conjectures discovered by Bloch and Srinivas \cite{BS} and extended by Colliot-Th\'el\`ene and Voisin \cite{CTV}.


In section \ref{s3}, we study when a positive semidefinite polynomial $f\in \R[X_1,\dots, X_n]$ is a sum of $2^n-1$ squares of rational functions.
We successively relate this property to the level of the function field $\R(Y)$ of the variety $Y:=\{Z^2+F=0\}$ in Proposition \ref{sumlevel} (this is due to Pfister), to degree $n$ unramified cohomology of $Y$ in Proposition \ref{levelnr} (an important tool is Voevodsky's solution  to the Milnor conjecture \cite{Voevodsky}) and to degree $n+1$ cohomology of $Y$ in Proposition \ref{nn+1} (this is the crucial step, that uses Bloch-Ogus theory, and where the rational connectedness of $Y$ plays a role).

Section \ref{s4} contains the cohomological computations on the variety $Y$ that are relevant to apply the results of section \ref{s3}. The last paragraph \ref{calculcohod2n} will only be useful when $n$ is odd and $d=2n$, and is complemented by a geometric coniveau estimate in Section \ref{s5}. The reader who is not interested in our partial and conditional results when $n\geq 3$ is odd and $d=2n$ may skip them.


Section \ref{s6} completes the proof of Theorem \ref{main}. For a generic choice of $f$ (that is when the degree of $f$ is maximal among the values allowed in the statement of Theorem \ref{main}, and $Y$ is a smooth variety), this is an immediate consequence of the results obtained so far.
In general, we do not know how to apply this argument directly, because we do not have a good control on the geometry of (a resolution of singularities of) $Y$. Instead, we rely on a specialization argument. This argument reduces Theorem \ref{main} to the generic case, but over a bigger real closed field. In particular, even if one is only interested in proving Theorem \ref{main} over $\mathbb{R}$, one has to work over real closed fields that are not necessarily archimedean.

\bigskip

{\it Acknowledgements.} I have benefited from numerous discussions with Olivier Wittenberg, that have shaped my understanding of the cohomology of real algebraic varieties, and have been very important for the completion of this work. 

I am grateful to Claire Voisin for explaining to me the coniveau computation contained in Section \ref{s5}, that allowed to deal with the $n=5$ and $d=10$ case of Theorem \ref{main}.

\section{Cohomology of real varieties}\label{s1}

Let $\R$ be a real closed field and $\C$ be an algebraic closure of $\R$. We will denote by $G:=\Gal(\C/\R)\simeq\mathbb{Z}/2\mathbb{Z}$ the Galois group. A variety over $\R$ is a separated scheme of finite type over $\R$.

\subsection{$2$-adic cohomology}

If $X$ is a variety over $\R$, we denote by $H^k(X,\mathbb{Z}/2^r\mathbb{Z}(j))$ its \'etale cohomology groups. These cohomology groups are finite: this follows from the Hochschild-Serre spectral sequence $$E_2^{p,q}=H^p(G,H^q(X_{\C},\mathbb{Z}/2^r\mathbb{Z}(j)))\Rightarrow H^{p+q}(X,\mathbb{Z}/2^r\mathbb{Z}(j))$$ using that $X_{\C}$ has finite cohomological dimension \cite[X Corollaire 4.3]{SGA43}, that the groups $H^q(X_{\C},\mathbb{Z}/2^r\mathbb{Z}(j))$ are finite \cite[XVI Th\'eor\`eme 5.1]{SGA43} and that a finite $G$-module has finite cohomology.

Let us define $H^k(X,\mathbb{Z}_2(j)):=\varprojlim_r H^k(X,\mathbb{Z}/2^r\mathbb{Z}(j))$. Since the Galois cohomology of finite $G$-modules is finite, \cite[(3.5) c)]{Jannsen} shows that these groups coincide with the continuous \'etale cohomology groups defined by Jannsen. In particular, we have a Hochschild-Serre spectral sequence \cite[(3.5) b)]{Jannsen}:
\begin{equation}\label{HS}
E_2^{p,q}=H^p(G,H^q(X_{\C},\mathbb{Z}_2(j)))\Rightarrow H^{p+q}(X,\mathbb{Z}_2(j)).
\end{equation}

We will also use freely the cup-products, cohomology groups with support, cycle class maps and Gysin morphisms defined by Jannsen \cite{Jannsen}.

Note that since $G=\mathbb{Z}/2\mathbb{Z}$, the sheaves $\mathbb{Z}/2^r\mathbb{Z}(j)$ only depend on the parity of $j$, hence so do all the cohomology groups considered above.

Let $\omega$ be the generator of $H^1(\R,\mathbb{Z}_2(1))\simeq\mathbb{Z}/2\mathbb{Z}$. We will denote as well by $\omega$ its reduction modulo $2$: the generator of $H^1(\R,\mathbb{Z}/2\mathbb{Z})\simeq\mathbb{Z}/2\mathbb{Z}$ . If $k\geq 1$, their powers $\omega^k$ generate $H^k(\R,\mathbb{Z}_2(k))\simeq\mathbb{Z}/2\mathbb{Z}$ and $H^k(\R,\mathbb{Z}/2\mathbb{Z})\simeq\mathbb{Z}/2\mathbb{Z}$, and we will still denote by $\omega^k$ their pull-backs to any variety $X$ over $\R$.

\subsection{Comparison with geometric cohomology.}\label{parrc}

Let $\pi:\Spec(\C)\to \Spec(\R)$ be the base-change morphism, and fix $j\in\mathbb{Z}$. There is a natural short exact sequence of \'etale sheaves on $\Spec(\R)$: $0\to\mathbb{Z}/2^r\mathbb{Z}(j)\to \pi_*\mathbb{Z}/2^r\mathbb{Z}\to\mathbb{Z}/2^r\mathbb{Z}(j+1)\to 0$, as one checks at the level of $G$-modules. They fit together to form a short exact sequence of $2$-adic sheaves on $\Spec(\R)$: $0\to\mathbb{Z}_2(j)\to \pi_*\mathbb{Z}_2\to\mathbb{Z}_2(j+1)\to 0$.

Let $X$ be a variety over $\R$, and let us still denote by $\pi:X_{\C}\to X$ the base-change morphism. Notice that by the Leray spectral sequence, $H^k(X, \pi_*\mathbb{Z}/2^r\mathbb{Z})=H^k(X_{\C},\mathbb{Z}/2^r\mathbb{Z})$. Now pull-back the exact sequence of $2$-adic sheaves $0\to\mathbb{Z}_2(j)\to \pi_*\mathbb{Z}_2\to\mathbb{Z}_2(j+1)\to 0$ on $X$ and take continuous \'etale cohomology. We obtain a long exact sequence: 
\begin{alignat}{2}\label{rc}
\dots\to H^k(X,\mathbb{Z}_2(j))\stackrel{\pi^*}{\longrightarrow}& H^k(X_{\C},\mathbb{Z}_2)\\
&\stackrel{\pi_*}{\longrightarrow}H^k(X,\mathbb{Z}_2(j+1))\stackrel{\omega}{\longrightarrow} H^{k+1}(X,\mathbb{Z}_2(j))\to\dots \nonumber
\end{alignat}
in which the boundary map $H^k(X,\mathbb{Z}_2(j+1))\to H^{k+1}(X,\mathbb{Z}_2(j))$ is the cup-product by the class of the extension $0\to\mathbb{Z}_2(j)\to \pi_*\mathbb{Z}_2\to\mathbb{Z}_2(j+1)\to 0$, that is the non-zero class $\omega\in H^1(G,\mathbb{Z}_2(1))\simeq\mathbb{Z}/2\mathbb{Z}$.


\subsection{Cohomological dimension}

Recall first the following well-known statement, that goes back to Artin:
\begin{prop}\label{realff}
Let $X$ be an integral variety over $\R$. The following are equivalent:
\begin{enumerate}[(i)]
\item $\R(X)$ is formally real, that is $-1$ is not a sum of squares in $\R(X)$.
\item $X$ has a smooth $\R$-point.
\item $X(\R)$ is Zariski-dense in $X$.
\end{enumerate}
\end{prop}

\begin{proof}
By the Artin-Lang homomorphism theorem \cite[Theorem 4.1.2]{BCR}, if (i) holds, every open affine subset of $X$ contains a $\R$-point, proving (iii). Conversely, if $X(\R)$ were Zariski-dense in $X$, $-1$ could not be a sum of squares in $\R(X)$, because we would get a contradiction by evaluating this identity at an $\R$-point outside of the poles of the rational functions that appear. That (ii) implies (iii) is a consequence of the implicit function theorem \cite[Corollary 2.9.8]{BCR}, and the converse is trivial.
\end{proof}

  From this proposition, it is possible to deduce estimates on the cohomological dimension of varieties $X$ over $\R$ without $\R$-points. For the cohomological dimension of $\R(X)$, this follows from a theorem of Serre \cite{Serreprofini} and Artin-Schreier theory. The cohomological dimension of an arbitrary variety $X$ may then be controlled using \cite[X Corollaire 4.2]{SGA43}.

Here, we rather point out places in the literature where the statements we need are explicitly formulated.

\begin{prop}\label{cd}
Let $X$ be an integral variety of dimension $n$ over $\R$ such that $X(\R)=\varnothing$. 
\begin{enumerate}[(i)]
\item $\R(X)$ has cohomological dimension $n$.
\item $X$ has \'etale cohomological dimension $\leq 2n$.
\item If $X$ is affine, $X$ has \'etale cohomological dimension $\leq n$.
\end{enumerate}
\end{prop}

\begin{proof}
 The first statement  is \cite[Proposition 1.2.1]{CTParimala}, where it is attributed to Ax.

The second (resp. third) statement follows from \cite[Corollary 7.21]{Scheiderer}, noticing that the real spectrum of $X$ is empty by Proposition \ref{realff} and using that $X_{\C}$ has \'etale cohomological dimension $\leq 2d$ (resp. $\leq d$) by \cite[X Corollaire 4.3]{SGA43} (resp. \cite[XIV Corollaire 3.2]{SGA43}).
\end{proof}

\section{Bloch-Ogus theory}\label{s2}

\subsection{Gersten's conjecture}\label{parBO}
In this paragraph, let $X$ be a smooth variety over $\R$.

  We want to apply Bloch-Ogus theory to the cohomology groups $H^k(X,\mathbb{Z}_2(j))$. For this purpose, one needs to check the validity of Gersten's conjecture for this cohomology theory. There are two ways to do so.

  First, the formal properties of continuous \'etale cohomology proven by Jannsen \cite{Jannsen} allow to prove that associating to a variety $X$ over $\R$ its continuous \'etale cohomology groups $H^k(X,\mathbb{Z}_2(j))$ is part of a Poincar\'e duality theory with supports in the sense of Bloch-Ogus \cite[Definition 1.3]{BO}, in the same way as it is proven for \'etale cohomology with finite coefficients by Bloch and Ogus in \cite[\S 2]{BO}. Then, it is possible to apply \cite[Theorem 4.2]{BO}.

  Another possibility is to use the axioms of \cite{BOG}, that are easier to check. That these axioms hold for continuous \'etale cohomology is explained for instance in \cite[\S 3C]{Kahn}, allowing to apply \cite[Corollary 5.1.11]{BOG}.

\vspace{1em}

  Let us now explain the meaning
of Gersten's conjecture in our context. Let us define $\mathcal{H}^k_X(j)$ to be the Zariski sheaf on $X$ that is the sheafification of $U\mapsto H^k(U,\mathbb{Z}_2(j))$. Moreover, if $z\in X$ is a point with closure $Z\subset X$, we define\footnote{
Beware that since continuous \'etale cohomology does not commute with inverse limit of schemes, this group does not coincide in general with the continuous Galois cohomology of the residue field of $z$.} \begin{equation}\label{cohinj}
H^k_{\to}(z,\mathbb{Z}_2(j)):=\varinjlim_{U\subset Z}H^k(U,\mathbb{Z}_2(j)),
\end{equation}
where $U$ runs over all nonempty open subsets of $Z$. We define $\iota_z:z\to X$ to be the inclusion, and we will consider the skyscraper sheaves $\iota_{z_{*}}H^k_{\to}(z,\mathbb{Z}_2(j))$ on $X$. 
Finally, we set $X^{(c)}$ to be the set of codimension $c$ points in $X$.
 Then the sheaves $\mathcal{H}^k_X(j)$ admit Cousin resolutions (see either \cite[(4.2.2)]{BO}, or \cite[Corollary 5.1.11]{BOG} taking into account purity \cite[(3.21)]{Jannsen} to obtain the precise form below):
\begin{alignat}{2}
0\to\mathcal{H}^k_X(j)&\to\bigoplus_{z\in X^{(0)}}\iota_{z_{*}}H^k_{\to}(z,\mathbb{Z}_2(j))\to\bigoplus_{z\in X^{(1)}}\iota_{z_{*}}H^{k-1}_{\to}(z,\mathbb{Z}_2(j-1))\label{Cousin}\\
&\to\dots\to\bigoplus_{z\in X^{(k)}}\iota_{z_{*}}H^0_{\to}(z,\mathbb{Z}_2(j-k))\to 0. \nonumber
\end{alignat}
The way this Cousin resolution is constructed, from a coniveau spectral sequence, shows that the arrows in (\ref{Cousin}) are given by maps in long exact sequences of cohomology with support, also called residue maps.

Since the sheaves in this resolution are flasque, the Cousin complex obtained by taking its global sections computes the Zariski cohomology of $\mathcal{H}^k_X(j)$.  For instance, this implies that
$H^0(X,\mathcal{H}^k_X(j))$ coincides with the unramified cohomology group $H^k_{\nr}(X,\mathbb{Z}_2(j))$, that is the subgroup of $H^k_{\to}(\eta,\mathbb{Z}_2(j))$ on which all residues at codimension $1$ points of $X$ vanish.



The exactness of (\ref{Cousin}) allows to compute the second page of the coniveau spectral sequence for $X$ mentioned above. As shown in \cite[Corollary 6.3]{BO} or \cite[Corollary 5.1.11]{BOG}, it reads:
\begin{equation}\label{coniveau}
E_2^{p,q}=H^p(X,\mathcal{H}_X^q(j))\Rightarrow H^{p+q}(X,\mathbb{Z}_2(j)).
\end{equation}
Recall that the filtration induced by this spectral sequence on $H^{k}(X,\mathbb{Z}_2(j))$ is the coniveau filtration, where a class $\alpha\in H^{k}(X,\mathbb{Z}_2(j))$ has coniveau $\geq c$ if it vanishes in the complement of a closed subset of codimension $c$ of $X$.


\subsection{Bloch-Ogus theory over $\R$.}







If $X$ is a variety over $\R$, we still denote by $\pi:X_{\C}\to X$ the natural morphism, and we view naturally $X_{\C}$ as a variety over $\R$. The following proposition was proved in \cite[Lemma 2.2.1]{CTS} over $\mathbb{R}$ and with $2$-torsion coefficients, but the proof goes through, and we include it for completeness.

\begin{prop}\label{propCTS}
Let $X$ be a smooth variety over $\R$ and fix $j\in\mathbb{Z}$. Then there exists a long exact sequence of Zariski sheaves on $X$:
\begin{equation}
\dots\to \mathcal{H}_X^k(j)\to \pi_* \mathcal{H}_{X_{\C}}^k\to \mathcal{H}_X^k(j+1)\to \mathcal{H}_X^{k+1}(j)\to\dots
\label{borc}
\end{equation}
Moreover, the sheaf $\pi_* \mathcal{H}_{X_{\C}}^k$ coincides with the sheafification of $U\mapsto H^k(U_{\C},\mathbb{Z}_2)$ and its cohomology groups are $H^q(X,\pi_* \mathcal{H}_{X_{\C}}^k)=H^q(X_{\C}, \mathcal{H}_{X_{\C}}^k)$ for any $k,q\geq 0$.
\end{prop}

\begin{proof}
Let $x\in X$. If $V$ is a neighbourhood of $\pi^{-1}(x)$ in $X_{\C}$, the sheaf $\mathcal{H}_{V}^k$ has a flasque Cousin resolution (\ref{Cousin}). Taking global sections and taking the limit over all such neighbourhoods $V$ gives a complex that is exact in positive degree (the argument for \'etale cohomology with finite coefficients is \cite[Proposition 2.1.2]{BOG}, and the corresponding effaceability condition for continuous \'etale cohomology follows from \cite[Theorem 5.1.10]{BOG}). As a consequence,
\begin{equation}\label{lim1}
\varinjlim_{V} H^p(V,\mathcal{H}_{V}^k)=0\text{ for }p>0.
\end{equation}
Considering the coniveau spectral sequences (\ref{coniveau}) for every $V$, and taking (\ref{lim1}) into account shows that 
\begin{equation}\label{lim2}
\varinjlim_{V} H^k(V,\mathbb{Z}_2)=\varinjlim_{V} H^0(V,\mathcal{H}_{V}^k).
\end{equation}
Note that in both (\ref{lim1}) and (\ref{lim2}), it is possible to restrict to neighbourhoods of the form $U_{\C}$, for $U\subset X$ because they form a cofinal family.

Now, the exact sequences obtained by applying (\ref{rc}) to all open subsets of $X$ fit together to induce a long exact sequence of Zariski presheaves on $X$. By exactness of sheafification, one obtains a long exact sequence of Zariski sheaves on $X$:
$$\dots\to \mathcal{H}_X^k(j)\to \mathcal{F}^k\to \mathcal{H}_X^k(j+1)\to \mathcal{H}_X^{k+1}(j)\to\dots,$$
where $\mathcal{F}^k$ is the sheafification of $U\mapsto H^k(U_{\C},\mathbb{Z}_2)$. The universal property of sheafification gives a morphism $\mathcal{F}^k\to\pi_* \mathcal{H}_{X_{\C}}^k$. The map induced on stalks at $x\in X$ is precisely (\ref{lim2}), hence an isomorphism. It follows that  $\mathcal{F}^k\simeq\pi_* \mathcal{H}_{X_{\C}}^k$, completing the construction of (\ref{borc}).

If $k\geq 0$ and $p>0$, the stalk of $R^p\pi_*\mathcal{H}_{X_{\C}}^k=0$ at $x\in X$ is given by (\ref{lim1}), hence trivial.
It follows that $R^p\pi_*\mathcal{H}_{X_{\C}}^k$ vanishes, and the Leray spectral sequence for $\pi$ implies the last statement of the proposition.
\end{proof}

\subsection{Consequences of the Bloch-Kato conjecture.}

The following proposition is due to Bloch and Srinivas \cite[Proof of Theorem 1]{BS} for $k\leq 2$ and to Colliot-Th\'el\`ene and Voisin \cite[Th\'eor\`eme 3.1]{CTV} in general. Since both references work over an algebraically closed field,  and since \cite{CTV} uses Betti cohomology, we repeat the proof to emphasize that it works in our setting.

\begin{prop}\label{BK}
Let $X$ be a smooth variety over $\R$. Then, for every $k\geq 0$, the sheaf $\mathcal{H}^{k+1}_X(k)$ is torsion free.
\end{prop}

\begin{proof}
%
%
Since it is a sheaf of $\mathbb{Z}_2$-modules, it suffices to prove that it has no $2$-torsion. Consider the exact sequence of $2$-adic sheaves on $X$: $0\to\mathbb{Z}_2(k)\stackrel{2}{\longrightarrow}\mathbb{Z}_2(k)\to\mu_2^{\otimes k}\to 0$. Taking long exact sequences of continuous cohomology over every open subset $U\subset X$ to get a long exact sequence of presheaves on $X$ and sheafifying it gives a long exact sequence of sheaves on $X$, part of which is:
\begin{equation*}
\mathcal{H}^{k}_X(k)\to\mathcal{H}^{k}_X(\mu_2^{\otimes k})\to\mathcal{H}^{k+1}_X(k)\stackrel{2}{\longrightarrow}\mathcal{H}^{k+1}_X(k),
\end{equation*}
where $\mathcal{H}^{k}_X(\mu_2^{\otimes k})$ is the sheafification of $U\mapsto H^k(U,\mu_2^{\otimes k})$.
Consequently, it suffices to prove the surjectivity of $\mathcal{H}^{k}_X(k)\to\mathcal{H}^{k}_X(\mu_2^{\otimes k})$.

On an open set $U\subset X$, the Kummer exact sequence $0\to\mu_2\to \mathbb{G}_m\stackrel{2}{\longrightarrow}\mathbb{G}_m\to 0$ induces a boundary map $H^0(U,\mathcal{O}_U^*)\to H^1(U,\mu_2)$. These maps sheafify to $\mathcal{O}_X^*\to \mathcal{H}^1_X(\mu_2)$, inducing via cup-products a morphism of sheaves $(\mathcal{O}_X^*)^{\otimes k}\to \mathcal{H}^{k}_X(\mu_2^{\otimes k})$. It is explained in \cite[end of section 2.2]{CTV} how Gersten's conjecture for Milnor K-theory proven by Kerz \cite{Kerz} and the Bloch-Kato conjecture proven by Rost and Voevodsky (since we only need this conjecture at the prime $2$, Voevodsky's work on Milnor's conjecture \cite[Corollary 7.4]{Voevodsky} is sufficient here) imply the surjectivity of this morphism.

Over an open set $U\subset X$, the boundary maps $H^0(U,\mathcal{O}_U^*)\to H^1(U,\mu_{2^r})$ for the Kummer exact sequences $0\to\mu_{2^r}\to \mathbb{G}_m\stackrel{2^r}{\longrightarrow}\mathbb{G}_m\to 0$ fit together to induce a map $H^0(U,\mathcal{O}_U^*)\to \varprojlim_r H^1(U,\mu_{2^r})=H^1(U,\mathbb{Z}_2(1))$. Again, this sheafifies to a morphism $\mathcal{O}_X^*\to \mathcal{H}^1_X(1)$, inducing via cup-products a morphism of sheaves $(\mathcal{O}_X^*)^{\otimes k}\to \mathcal{H}^{k}_X(k)$ lifting $(\mathcal{O}_X^*)^{\otimes k}\to \mathcal{H}^{k}_X(\mu_2^{\otimes k})$. The surjectivity of $\mathcal{H}^{k}_X(k)\to\mathcal{H}^{k}_X(\mu_2^{\otimes k})$ is now a consequence of the surjectivity of 
$(\mathcal{O}_X^*)^{\otimes k}\to \mathcal{H}^{k}_X(\mu_2^{\otimes k})$.
\end{proof}

In \cite{BS} and \cite{CTV}, the authors worked over an algebraically closed field, and the Tate twist was not essential for the result to hold. Here, it is very important: it is not true in general that the sheaf $\mathcal{H}^{k}_X(k)$ has no torsion.

As in these references, straightforward corollaries are:

\begin{cor}\label{sstorsion}
Let $X$ be a smooth variety over $\R$ and $k\geq 0$.

 Then $H^{k+1}_{\nr}(X,\mathbb{Z}_2(k))=H^0(X,\mathcal{H}^{k+1}_X(k))$ is torsion free.
\end{cor}

\begin{cor}\label{sstorsionbis}
Let $X$ be an integral variety over $\R$ with generic point $\eta$ and $k\geq 0$.

 Then $H^{k+1}_{\to}(\eta,\mathbb{Z}_2(k))$ is torsion free.
\end{cor}

\begin{proof}
If $\alpha\in H^{k+1}(U,\mathbb{Z}_2(k))$ is a torsion class on a smooth open subset $U\subset X$, it vanishes in $H^{k+1}_{\nr}(U,\mathbb{Z}_2(k))$ by 
Corollary \ref{sstorsion}, hence on an open subset $V\subset U$.
\end{proof}

Another application of Proposition \ref{BK} is:

\begin{prop}\label{bkrc}
Let $X$ be a smooth variety over $\R$. Then for every $k\geq 0$, there is an exact sequence:
\begin{equation*}
0\to \mathcal{H}_X^{k-1}(k)\to \pi_* \mathcal{H}_{X_{\C}}^{k-1}\to \mathcal{H}_X^{k-1}(k+1)\to \mathcal{H}_X^{k}(k)\to \pi_* \mathcal{H}_{X_{\C}}^{k}\to \mathcal{H}_X^{k}(k+1)\to 0.
\end{equation*}
\end{prop}

\begin{proof}
Let us prove that the long exact sequence (\ref{borc}) splits into these shorter exact sequences. It suffices to prove that, for $k\geq 0$, the morphism  $\mathcal{H}_X^{k-1}(k)\to \pi_* \mathcal{H}_{X_{\C}}^{k-1}$ is injective.
The composition $\mathcal{H}_X^{k-1}(k)\to \pi_* \mathcal{H}_{X_{\C}}^{k-1}\to\mathcal{H}_X^{k-1}(k)$ is multiplication by $2$.
Consequently, the kernel of $\mathcal{H}_X^{k-1}(k)\to \pi_* \mathcal{H}_{X_{\C}}^{k-1}$ is of $2$-torsion. Since $\mathcal{H}_X^{k-1}(k)$ is torsion free by Proposition \ref{BK}, this kernel is trivial as required.
\end{proof}

\begin{prop}\label{bkrc2}
Let $X$ be an integral variety over $\R$ with generic point $\eta$. Then for every $k\geq 0$, there is an exact sequence:
\begin{alignat*}{2}
0\to H^{k-1}_{\to}&(\eta,\mathbb{Z}_2(k))\to H^{k-1}_{\to}(\eta,\pi_*\mathbb{Z}_2)\to H^{k-1}_{\to}(\eta,\mathbb{Z}_2(k+1))\\
&\to H^{k}_{\to}(\eta,\mathbb{Z}_2(k))\to H^{k}_{\to}(\eta,\pi_*\mathbb{Z}_2)\to H^{k}_{\to}(\eta,\mathbb{Z}_2(k+1))\to 0.
\end{alignat*}
\end{prop}

\begin{proof}
Take the direct limit of the long exact sequence (\ref{rc}) applied to all open subsets of $X$: it splits into exact sequences of length six by the same argument as in the proof of Proposition \ref{bkrc}, using Corollary \ref{sstorsionbis} instead of Corollary \ref{sstorsion}.
\end{proof}

\section{Sums of squares and unramified cohomology}\label{s3}

\subsection{Sums of squares and level.}\label{parsquares}

Let $n\geq 1$, consider a nonzero positive semidefinite polynomial $f\in \R[X_1,\dots,X_n]$ and its homogenization $F\in \R[X_0,\dots,X_n]$. Notice that since an odd degree polynomial over $\R$ changes sign, $f$ and $F$ must have even degree. This allows to consider the double cover $Y$ of $\mathbb{P}^n_{\R}$ ramified over $\{F=0\}$ defined by the equation $Y:=\{Z^2+F=0\}$ in the weighted projective space $\mathbb{P}(1,\dots, 1,\deg(F)/2)$. 

\begin{lem}\label{rev}
The variety $Y$ is integral, $\R(Y)$ is not formally real,
and if $\widetilde{Y}\to Y$ is a resolution of singularities, $\widetilde{Y}(\R)=\varnothing$.
\end{lem}

\begin{proof}
To prove that $Y$ is integral, one has to check that $-f$ is not a square in $\R(X_1,\dots,X_n)$, equivalently that it is not a square in $\R[X_1,\dots,X_n]$. But if it were, $f$ would be negative on $\R^n$, hence zero on $\R^n$ by positivity, hence zero by Zariski-density of $\R^n$ in $\C^n$: this is a contradiction.

The $\R$-points of $\widetilde{Y}$ necessarily lie above $\R$-points of $Y$, hence, by positivity of $F$, above zeroes of $F$. Consequently,  $\widetilde{Y}(\R)$ is not Zariski-dense in $\widetilde{Y}$. Applying Proposition \ref{realff} using the smoothness of $\widetilde{Y}$ shows that $\widetilde{Y}(\R)=\varnothing$, and that $\R(Y)$ is not formally real.
\end{proof}

Recall that the level $s(K)\in\mathbb{N}^*\cup\{\infty\}$ of a field $K$ is $\infty$ if $-1$ is not a sum of squares in $K$ and the smallest $s$ such that $-1$ is a sum of $s$ squares otherwise. In the latter case, it has been shown by Pfister \cite[Satz 4]{PfisterStufe} to be a power of $2$.

\begin{prop}\label{sumlevel}
The polynomial $f$ is a sum of $2^n-1$ squares in $\R(X_1,\dots,X_n)$ if and only if $\R(Y)$ has level $< 2^{n}$.
Conversely, the polynomial $f$ is not a sum of $2^n-1$ squares in $\R(X_1,\dots,X_n)$ if and only if $\R(Y)$ has level $2^{n}$.
\end{prop}

\begin{proof}
Proposition \ref{realff} shows that $\R(X_1,\dots,X_n)$ is formally real and Artin's solution to Hilbert's 17\textsuperscript{th} problem \cite{Artin17} shows that $f$ is a sum of squares in $\R(X_1,\dots,X_n)$.

Then, \cite[Chap. 11 Theorem 2.7]{Lam} applies and shows that $f$ is a sum of $2^n-1$ squares in $\R(X_1,\dots,X_n)$ if and only if $\R(Y)$ has level $< 2^{n}$ (this is essentially due to Pfister: the statement we have used is very close and its proof is identical to \cite[Satz 5]{PfisterStufe}).

Since $\R(Y)$ is not formally real by Lemma \ref{rev}, Pfister has shown that its level is $\leq 2^n$ \cite[Theorem 2]{Pfister17}. This concludes the proof.
\end{proof}

\subsection{Level and unramified cohomology}
To apply Proposition \ref{sumlevel}, we need to control the level of the function field of a variety over $\R$. The following proposition relates it to one of its unramified cohomology groups.
The equivalence (i)$\Leftrightarrow$(ii) 
is hinted at in \cite[bottom of p.236]{CTNL}, at least for $n=3$.
I am grateful to Olivier Wittenberg for explaining to me that the implication (ii)$\Rightarrow$(iii) holds.

\begin{prop}\label{levelnr}
Let $X$ be a smooth integral variety over $\R$
, and fix $n\geq 1$.  The following assertions are equivalent:
\begin{enumerate}[(i)]
\item The function field $\R(X)$ has level $<2^{n}$.
\item The map $H^n(\R,\mathbb{Z}/2\mathbb{Z})\to H^n(\R(X),\mathbb{Z}/2\mathbb{Z})$ vanishes.
\item The map $H^n(\R,\mathbb{Z}_2(n))\to H^n_{\nr}(X,\mathbb{Z}_2(n))$ vanishes.
\end{enumerate}
\end{prop}

\begin{proof}
Consider the property that the level of $\R(X)$ is $<2^n$.  It is equivalent to the fact $-1$ is a sum of $2^n-1$ squares in $\R(X)$, hence to the fact that the Pfister quadratic form $q:=\langle 1,1\rangle^{\otimes n}$ is isotropic over $\R(X)$. By a theorem of Elman and Lam \cite[Corollary 3.3]{ElmanLam}, this is equivalent to the vanishing of the symbol $\{-1\}^{n}$ in the Milnor K-theory group $K_n^M(\R(X))/2$. By Voevodsky's proof of the Milnor conjecture \cite[Corollary 7.4]{Voevodsky}, the natural map $K_n^M(\R(X))/2\to H^n(\R(X),\mathbb{Z}/2\mathbb{Z})$ is an isomorphism, so that our property is equivalent to the vanishing of $H^n(\R,\mathbb{Z}/2\mathbb{Z})\to H^n(\R(X),\mathbb{Z}/2\mathbb{Z})$. We have proven that (i) and (ii) are equivalent.

Suppose that (iii) holds and let $\eta$ be the generic point of $X$.
 The definition of $H^n_{\nr}(X,\mathbb{Z}_2(n))$ as a subgroup of $H^n_{\to}(\eta,\mathbb{Z}_2(n))$ shows that 
$H^n(\R,\mathbb{Z}_2(n))\to H^n_{\to}(\eta,\mathbb{Z}_2(n))$ vanishes.
 Then we have a commutative diagram:
\begin{equation*}
\xymatrix{
H^n(\R,\mathbb{Z}_2(n))\ar^{\cong}[d]\ar[r]
&  H^n_{\to}(\eta,\mathbb{Z}_2(n))\ar[d]  \\
H^n(\R,\mathbb{Z}/2\mathbb{Z})\ar[r]
& H^n_{\to}(\eta,\mathbb{Z}/2\mathbb{Z}),
}
\end{equation*}
where the groups on the right are defined as inductive limits on the open subsets of $X$ as in (\ref{cohinj}), showing that $H^n(\R,\mathbb{Z}/2\mathbb{Z})\to H^n_{\to}(\eta,\mathbb{Z}/2\mathbb{Z})$ vanishes. Since \'etale cohomology commutes with such limits \cite[VII Corollaire 5.8]{SGA42}, $H^n_{\to}(\eta,\mathbb{Z}/2\mathbb{Z})$ is nothing but the Galois cohomology group $H^n(\R(X),\mathbb{Z}/2\mathbb{Z})$, proving (ii).


Suppose conversely that (ii) holds, and let $U\subset X$ be an open subset such that  $\omega^{n}$ vanishes in $H^n(U,\mathbb{Z}/2\mathbb{Z})$. 
Consider the following commutative exact diagram, where the lines are (\ref{rc}):
\begin{equation*}
\xymatrix{
H^{n-1}(U,\mathbb{Z}_2(n-1))\ar^-{\omega}[r]\ar^2[d] &H^{n}(U,\mathbb{Z}_2(n))\ar^2[d]\ar[r]&H^{n}(U_{\C},\mathbb{Z}_2)\ar^2[d]\\
H^{n-1}(U,\mathbb{Z}_2(n-1))\ar^-{\omega}[r]&H^{n}(U,\mathbb{Z}_2(n))\ar[r]\ar[d]&H^{n}(U_{\C},\mathbb{Z}_2)\\
&H^{n}(U,\mathbb{Z}/2\mathbb{Z})&\\
}
\end{equation*}
Look at $\omega^{n}\in H^{n}(U,\mathbb{Z}_2(n))$. By hypothesis, it vanishes in $H^{n}(U,\mathbb{Z}/2\mathbb{Z})$, hence may be written $2\alpha$ for some $\alpha\in H^{n}(U,\mathbb{Z}_2(n))$. Since $\omega^n\in H^{n}(U,\mathbb{Z}_2(n))$ is the image of $\omega^{n-1}\in H^{n-1}(U,\mathbb{Z}_2(n-1))$, $\alpha_{\C}\in H^{n}(U_{\C},\mathbb{Z}_2)$ is a $2$-torsion class.
By Corollary \ref{sstorsionbis}, any torsion class in $H^{n}(U_{\C},\mathbb{Z}_2)$ vanishes on an open subset: up to shrinking $U$, we may assume that $\alpha_{\C}=0$, hence that there is $\beta\in H^{n-1}(U,\mathbb{Z}_2(n-1))$ such that $\beta\cdot\omega=\alpha$. Then $\omega^n=\beta\cdot 2\omega=0\in H^{n}(U,\mathbb{Z}_2(n))$, proving (iii).
\end{proof}

\subsection{From degree $n$ to degree $n+1$ cohomology}\label{parnn+1}

Condition (iii) in Proposition \ref{levelnr} means that $\omega^n$
has coniveau $\geq 1$.
Proposition \ref{nn+1} uses Bloch-Ogus theory to relate this property to the coniveau of $\omega^{n+1}$.


\vspace{1em}

Fix $n\geq 1$ and let $X$ be a smooth variety 
over $\R$.
The coniveau spectral sequence (\ref{coniveau}) induces two maps $H^n(X,\mathbb{Z}_2(n))\stackrel{\phi}{\longrightarrow} H^n_{\nr}(X,\mathbb{Z}_2(n))$ and 
$K:=\Ker[H^{n+1}(X,\mathbb{Z}_2(n+1))\to H^{n+1}_{\nr}(X,\mathbb{Z}_2(n+1))]\stackrel{\psi}{\longrightarrow}H^1(X,\mathcal{H}_X^n(n+1))$.

Cup-product with 
$\omega$
gives morphisms $H^n(X,\mathbb{Z}_2(n))\stackrel{\omega}{\longrightarrow}H^{n+1}(X,\mathbb{Z}_2(n+1))$ and $H^n_{\nr}(X,\mathbb{Z}_2(n))\stackrel{\omega}{\longrightarrow}H^{n+1}_{\nr}(X,\mathbb{Z}_2(n+1))$. Let $I:=\{\alpha\in H^n(X,\mathbb{Z}_2(n))\mid \alpha\cdot\omega\in K\}$ and
$I_{\nr}:=\{\alpha\in H^n_{\nr}(X,\mathbb{Z}_2(n))\mid \alpha\cdot\omega=0\}$.

Finally, Proposition \ref{bkrc} gives an exact sequence of sheaves on $X$:
\begin{equation}\label{exactn}
0\to\mathcal{H}_X^n(n+1)\to\pi_*\mathcal{H}_{X_{\C}}^n\to
\mathcal{H}_X^{n}(n)\stackrel{\omega}{\longrightarrow}\mathcal{H}_X^{n+1}(n+1)\to\cdots
\end{equation}
Taking cohomology, we obtain an exact sequence:
\begin{equation}\label{noyaudelta}
0\to H^n_{\nr}(X,\mathbb{Z}_2(n+1))\to H^n_{\nr}(X_{\C},\mathbb{Z}_2)\to I_{\nr}\stackrel{\delta}{\longrightarrow}H^1(X,\mathcal{H}_X^n(n+1)).
\end{equation}

\begin{lem}\label{commute} Let $X$ be a smooth variety
over $\R$.
The diagram 
\begin{equation*}
\xymatrix{
I \ar[r]^-{\omega}\ar[d]_-{\phi}
&  K\ar[d]^-{\psi}  \\
I_{\nr}\ar[r]^-{\delta}
& H^1(X,\mathcal{H}_X^n(n+1))
}
\end{equation*}
constructed above commutes.
\end{lem}

\begin{proof}
Let $\alpha\in I$. By hypothesis, the class $\alpha\cdot\omega\in H^{n+1}(X,\mathbb{Z}_2(n+1))$ vanishes on an open subset $U\subset X$. Let $D:=X\setminus U$ be endowed with its reduced structure.

The description of $H^1(X,\mathcal{H}_X^n(n+1))$ as a cohomology group of the Cousin complex (\ref{Cousin}) shows that if $X^{\circ}\subset X$ is an open subset whose complement has codimension $\geq 2$, the restriction $H^1(X,\mathcal{H}_X^n(n+1))\to H^1(X^{\circ},\mathcal{H}_{X^{\circ}}^n(n+1))$ is injective. Consequently, to prove that $\psi(\alpha\cdot\omega)=\delta\circ\phi(\alpha)$, it is possible to remove from $X$ a closed subset of codimension $\geq 2$. 
This allows to suppose that $D$ is smooth of pure codimension $1$.

  Our next task is to identify concretely $\delta\circ\phi(\alpha)$. The cohomology theory with supports in the sense of \cite[Definition 5.1.1]{BOG} that to a variety $X$ over $\R$ and a closed subset $Z\subset X$ associates the groups $H_Z^k(X,\pi_*\mathbb{Z}_2)=H^k_{Z_{\C}}(X_{\C},\mathbb{Z}_2)$ satisfies axioms COH1 and COH3 by \cite[5.5 (1)]{BOG}, hence COH2 by \cite[Proposition 5.3.2]{BOG}. It follows from \cite[Corollary 5.1.11]{BOG} that the sheafification of $U\mapsto H^n(U_{\C},\mathbb{Z}_2)$ (that is $\pi_*\mathcal{H}^n_{X_{\C}}$ by Proposition \ref{propCTS}) admits a Cousin resolution by flasque sheaves, and the same goes for $\pi_*\mathcal{H}^{n+1}_{X_{\C}}$. These resolutions fit together with the Cousin resolutions (\ref{Cousin}) of $\mathcal{H}_X^n(n+1)$, $\mathcal{H}_X^n(n)$, $\mathcal{H}_X^{n+1}(n+1)$ and $\mathcal{H}_X^{n+1}(n)$, giving rise to a diagram, that is an exact sequence of flasque resolutions for the exact sequence of sheaves (\ref{exactn}) by Proposition \ref{bkrc2}. Let us only draw the relevant part of the diagram, containing the Cousin resolutions for $\mathcal{H}_X^n(n+1)$, $\pi_*\mathcal{H}^{n+1}_{X_{\C}}$ and $\mathcal{H}_X^n(n)$:

\begin{equation}
\begin{gathered}\label{Cousindiag}
\xymatrix{
\bigoplus\limits_{z\in X^{(0)}}\iota_{z_{*}}H^n_{\to}(z,\mathbb{Z}_2(n+1))\ar[r]\ar[d]& \bigoplus\limits_{z\in X^{(1)}}\iota_{z_{*}}H^{n-1}_{\to}(z,\mathbb{Z}_2(n)) \ar[r]\ar[d] &\dots\\
\bigoplus\limits_{z\in X^{(0)}}\iota_{z_{*}}H^n_{\to}(z,\pi_*\mathbb{Z}_2)\ar[r]\ar[d] &\bigoplus\limits_{z\in X^{(1)}}\iota_{z_{*}}H^{n-1}_{\to}(z,\pi_*\mathbb{Z}_2) \ar[r]\ar[d]& \dots\\
\bigoplus\limits_{z\in X^{(0)}}\iota_{z_{*}}H^n_{\to}(z,\mathbb{Z}_2(n))\ar[r] &\bigoplus\limits_{z\in X^{(1)}}\iota_{z_{*}}H^{n-1}_{\to}(z,\mathbb{Z}_2(n-1)) \ar[r]& \dots\\
}
\end{gathered}
\end{equation}



It is now possible to give a description of $\delta\circ\phi(\alpha)$ by a diagram chase in the diagram obtained by taking the global sections of (\ref{Cousindiag}). More precisely, $\alpha$ induces a class $\phi(\alpha)\in \Ker[\bigoplus_{z\in X^{(0)}}H^n_{\to}(z,\mathbb{Z}_2(n))\to\bigoplus_{z\in X^{(1)}}H^{n-1}_{\to}(z,\mathbb{Z}_2(n-1))]$. Lifting it in $\bigoplus_{z\in X^{(0)}}H^n_{\to}(z,\pi_*\mathbb{Z}_2)$ by the hypothesis that $\alpha\in I$, pushing it to $\bigoplus_{z\in X^{(1)}}H^{n-1}_{\to}(z,\pi_*\mathbb{Z}_2)$ and lifting it again to 
$\bigoplus_{z\in X^{(1)}}H^{n-1}_{\to}(z,\mathbb{Z}_2(n))$ gives a cohomology class of degree one of the complex of global sections of the Cousin resolution of $\mathcal{H}_X^n(n+1)$ representing $\delta\circ\phi(\alpha)\in H^1(X,\mathcal{H}_X^n(n+1))$.

At this point, consider the following commutative diagram, whose rows are exact sequences of cohomology with support, whose columns are instances of (\ref{rc}), and where the coefficient ring $\mathbb{Z}_2$ has been omitted:
\begin{equation*}
\xymatrix{
 & &H^n(U,n+1)\ar[r]\ar[d]&H^{n-1}(D,n)\ar[d]^{\pi^*}\\
& &H^{n}(U_{\C})\ar[d]^{\pi_*}\ar[r]^{\partial}&H^{n-1}(D_{\C})\ar[d]\\
H^{n-2}(D,n-1)\ar[d]\ar[r]& H^{n}(X,n)\ar^{\omega}[d]\ar[r]^{j^*}&H^{n}(U,n)\ar[d]\ar[r]&H^{n-1}(D,n-1)&\\
H^{n-1}(D,n)\ar[d]\ar[r]^{i_*}& H^{n+1}(X,n+1)\ar[r]&H^{n+1}(U,n+1)&\\
H^{n-1}(D_{\C})& &&\\
}
\end{equation*}
Here, we have denoted by $i:D\to X$ and $j:U\to X$ the inclusions, and by $\partial$ the residue map.
By our choice of $U$, $\alpha\in H^n(X,\mathbb{Z}_2(n))$ vanishes in $H^{n+1}(U,\mathbb{Z}_2(n+1))$. Chasing the diagram, there are two ways to construct a (not well-defined) class in $H^{n-1}(D,\mathbb{Z}_2(n))$. First, we may consider a class $\beta\in H^{n-1}(D,\mathbb{Z}_2(n))$ such that $i_*\beta=\alpha\cdot\omega$. Second, we may lift $j^*\alpha$ along $\pi_*$, apply the residue map $\partial$, and lift the resulting class along $\pi^*$ to obtain $\gamma\in H^{n-1}(D,\mathbb{Z}_2(n))$.

Our diagram has been constructed from the diagram of distinguished triangles in the derived category of $2$-adic sheaves on $X$:
\begin{equation*}
\xymatrix{
 i_*Ri^!\mathbb{Z}_2(n+1)\ar[d]\ar[r]&\mathbb{Z}_2(n+1)\ar[d]\ar[r]&Rj_*j^*\mathbb{Z}_2(n+1)\ar[d]\ar[r]&\\
 i_*Ri^!\pi_*\mathbb{Z}_2\ar[d]\ar[r]&\pi_*\mathbb{Z}_2\ar[d]\ar[r]&Rj_*j^*\pi_*\mathbb{Z}_2\ar[d]\ar[r]&\\
 i_*Ri^!\mathbb{Z}_2(n)\ar[d]\ar[r]&\mathbb{Z}_2(n)\ar[d]\ar[r]&Rj_*j^*\mathbb{Z}_2(n)\ar[d]\ar[r]&\\
&&&\\
}
\end{equation*}
  A homological algebra lemma due to Jannsen \cite[Lemma p. 268]{Jannsenletter}, applied exactly as in \cite[Proof of Theorem 2]{Jannsenletter}, shows that the images of $\beta$ and $\gamma$ in $H^{n-1}(D_{\C},\mathbb{Z}_2)$, that are well-defined up to the image of $H^n(U,\mathbb{Z}_2(n+1))$, coincide up to a sign. It follows that $\beta$ and $\gamma$, well-defined up to the images of $H^n(U,\mathbb{Z}_2(n+1))$ and $H^{n-2}(D,\mathbb{Z}_2(n-1))$ in $H^{n-1}(D,\mathbb{Z}_2(n))$, coincide up to a sign.

Now notice that $\beta$ and $\gamma$ induce classes in $\bigoplus_{z\in X^{(1)}}H^{n-1}_{\to}(z,\mathbb{Z}_2(n))$. Our explicit description of $\delta\circ\phi(\alpha)$, shows that $\beta$ is a representative of it as a cohomology class of degree one of the Cousin complex.
On the other hand, $\gamma$ has been constructed by lifting $\alpha\cdot\omega$ along the Gysin morphism $H^{n-1}(D,\mathbb{Z}_2(n))\to H^{n+1}(X,\mathbb{Z}_2(n+1))$. By construction of the coniveau spectral sequence (\cite[\S 3]{BO}, \cite[\S 1]{BOG}), $\gamma$ 
is a representative of $\psi(\alpha\cdot\omega)$ as a cohomology class of degree one of the Cousin complex.

At this point, we have proven that $\psi(\alpha\cdot\omega)=[\gamma]=-[\beta]=-\delta\circ\phi(\alpha)$. Since this element is $2$-torsion because $\omega$ is, one has in fact $\psi(\alpha\cdot\omega)=\delta\circ\phi(\alpha)$, as wanted.
\end{proof}

\begin{prop}\label{nn+1}
Let $X$ be a smooth projective variety over $\R$, and fix $n\geq 1$.
Consider the following assertions:
\begin{enumerate}[(i)]
\item The class $\omega^n\in H^n(X,\mathbb{Z}_2(n))$ has coniveau $\geq 1$.
\item The class $\omega^{n+1}\in H^{n+1}(X,\mathbb{Z}_2(n+1))$ has coniveau $\geq 2$.
\end{enumerate}
Then (i) implies (ii). Moreover, if $\CH_0(X_{\C})$ is supported on a closed subvariety of $X_{\C}$ of dimension $n-1$, the converse holds.
\end{prop}

\begin{proof}
Either (i) or (ii) implies that $\omega^{n+1}$ has coniveau $\geq 1$, or equivalently that it vanishes in $H^{n+1}_{\nr}(X,\mathbb{Z}_2(n+1))$. Let  us suppose this is the case: in particular, $\omega^n\in I$. 

 By the coniveau spectral sequence (\ref{coniveau}), $\omega^n$ has coniveau $\geq 1$ in $X$ if and only if its class in $H^{n}_{\nr}(X,\mathbb{Z}_2(n))$ vanishes, and $\omega^{n+1}$ has coniveau $\geq 2$ if and only if its class in $H^1(X,\mathcal{H}_X^n(n+1))$ vanishes. Then consider the diagram:
\begin{equation*}
\xymatrix{
H^n(\R,\mathbb{Z}_2(n)) \ar[r]^-{\omega}_-{\cong}\ar[d]
&  H^{n+1}(\R,\mathbb{Z}_2(n+1))\ar[d]  \\
I_{\nr}\ar[r]^-{\delta}
& H^1(X,\mathcal{H}_X^n(n+1)),
}
\end{equation*}
that is commutative by Lemma \ref{commute}.
Contemplating it shows that (i) implies (ii). 

Conversely, if $CH_0(X_{\C})$ is supported on a closed subvariety of $X_{\C}$ of dimension $n-1$, we have $H^n_{\nr}(X_{\C},\mathbb{Z}_2)=0$ by \cite[Proposition 3.3 (ii)]{CTV}. Indeed, the argument given there for Betti cohomology over $\mathbb{C}$, that relies on decomposition of the diagonal, works as well for $2$-adic cohomology over $\C$. It then follows from the exact sequence (\ref{noyaudelta}) that $\delta$ is injective, proving that (ii) implies (i).
\end{proof}

\section{Cohomology of smooth double covers}\label{s4}

Recall the notation of paragraph \ref{parsquares}. The polynomial $F\in\R[X_0,\dots,X_n]$ is the homogenization of a nonzero positive semidefinite polynomial $f\in \R[X_1,\dots,X_n]$. Its degree $d$ is even. We introduced the double cover $Y$ of $\mathbb{P}^n_{\R}$ ramified over $\{F=0\}$ defined by the equation $Y:=\{Z^2+F=0\}$. 
  
In all this section, we make the additional hypothesis that $\{F=0\}$ is smooth so that $Y$ is smooth. By Lemma \ref{rev}, $Y(\R)=\varnothing$. The main goal of this section is to prove Propositions \ref{vanisheven2} and \ref{combiodd}.

\subsection{Geometric cohomology}
We first collect the results on the cohomology of $Y_{\C}$ that we will need. They follow from general theorems on the cohomology of weighted complete intersections due to Dimca \cite{Dimca}. When $n=3$, we could also have applied \cite[Corollary 1.19 and Lemma 1.23]{Clemens}.

\begin{prop}\label{cohogeo}
 Let $H_{\C}\in H^2(Y_{\C},\mathbb{Z}_2(1))$ be the class of $\mathcal{O}_{\mathbb{P}^n_{\C}}(1)$.
\begin{enumerate}[(i)]
\item The cohomology groups $H^k(Y_{\C},\mathbb{Z}_2)$ have no torsion.
\item If $k\neq n$ is odd, $H^k(Y_{\C},\mathbb{Z}_2)=0$.
\item If $0\leq l< n/2$,  $H^{2l}(Y_{\C},\mathbb{Z}_2)\simeq\mathbb{Z}_2(l)$ as a $G$-module, and is generated by $H_{\C}^l$.
\item If $n/2<l\leq n$,  $H^{2l}(Y_{\C},\mathbb{Z}_2)\simeq\mathbb{Z}_2(l)$ as a $G$-module, and has a generator $\alpha_{l}$ such that $2\alpha_{l}=H_{\C}^l$.
\end{enumerate}
\end{prop}

\begin{proof}
It suffices to prove the equalities as $\mathbb{Z}_2$-modules (this means that is it is possible to forget the twist indicating the action of $G$), because one recovers the correct twist by noticing that the relevant cohomology groups are rationally generated by algebraic cycles.

Using the fact that $Y_{\C}$ is defined over an algebraically closed subfield that may be embedded in $\mathbb{C}$  together with the invariance of \'etale cohomology under an extension of algebraically closed fields, it suffices to prove the lemma when $\C=\mathbb{C}$. Moreover, by comparison with Betti cohomology, it suffices to prove it for Betti cohomology. 

Since $Y_{\C}$ is a strongly smooth weighted complete intersection in the sense of \cite{Dimca}, its cohomology groups have no torsion by \cite[Proposition 6 (ii)]{Dimca}. Moreover, its Betti numbers in degree $k\neq n$ are computed in \cite[Proposition 6 (i)]{Dimca}.

If $l<n/2$, the class $H_{\C}^l\in H^{2l}(Y_{\C},\mathbb{Z}_2)$ cannot be divisible by $2$ because the intersection product $\frac{H_{\C}^l}{2}\cdot\frac{H_{\C}^l}{2}\cdot H_{\C}^{n-2l}=\frac{1}{2}$ would not be an integer. It follows that $H_{\C}^l$ generates $H^{2l}(Y_{\C},\mathbb{Z}_2)$.

If $l>n/2$, since $H_{\C}^l\cdot H_{\C}^{n-l}=2$, it follows by Poincar\'e duality that $H^{2l}(Y_{\C},\mathbb{Z}_2)$ is generated by a class $\alpha_l$ such that $2\alpha_l=H_{\C}^l$.
\end{proof}

\subsection{Preparation for a deformation argument}\label{parsemialg}

In the next paragraphs, we will perform some computations on the cohomology of $Y$. One of the arguments we will use, in the proofs of Lemma \ref{vanisheven1} and Proposition \ref{combiodd}, is a reduction to the Fermat double cover $Y^{\dagger}:=\{Z^2+F^{\dagger}=0\}$, where $F^{\dagger}:=X_0^d+\dots+X_n^d$ is the Fermat equation. This deformation argument relies on a little bit of semi-algebraic geometry. We have found it more convenient to collect here the relevant lemmas.

Let $V:=\R[X_0,\dots,X_n]_d$ be space of the degree $d$ homogeneous polynomials viewed as an algebraic variety over $\R$.
The discriminant $\Delta\subset V$ is the closed algebraic subvariety parametrizing equations that do not define smooth hypersurfaces in $\mathbb{P}^n_{\R}$. It is irreducible, and a general point of $\Delta$ defines a hypersurface with only one ordinary double point as singularities. Let $\Delta'\subset\Delta$ be the closed algebraic subvariety parametrizing singular hypersurfaces that do not have only one ordinary double point as singularities: it has codimension $\geq 2$ in $V$. We view the sets of $\R$-points $V(\R)$, $\Delta(\R)$ and $\Delta'(\R)$ as semi-algebraic sets. Define:
$$\Pi:=\{H\in V(\R)\mid H(x_0,\dots,x_n)>0 \text{ for every } (x_0,\dots,x_n)\in \R^{n+1}\setminus(0,\dots,0)\}.$$

\begin{lem}\label{open}
The set $\Pi\subset V(\R)$ is convex, open and semi-algebraic. Moreover, the polynomials $F$ and $F^{\dagger}$ belong to $\Pi$.
\end{lem}

\begin{proof}
It is immediate that $\Pi$ is convex.
We will rather prove that the complement of $\Pi$ is a closed semi-algebraic set.
By homogeneity of $H$, it coincides with the projection to $V(\R)$ of:
$$Q:=\{(H,x_0,\dots,x_n)\in V(\R)\times\mathbb{S}^n\mid H(x_0,\dots,x_n)\leq 0\},$$
where $\mathbb{S}^n:=\{(x_0,\dots,x_n)\in \R^{n+1}\mid x_0^2+\dots+x_n^2=1\}$ is the unit sphere.

That it is semi-algebraic follows from the Tarski-Seidenberg theorem \cite[Theorem 2.2.1]{BCR}. To check that it is closed, it suffices to check that its intersection with every closed hypercube in $V(\R)$ is closed, which follows from \cite[Theorem 2.5.8]{BCR}.

That $F^{\dagger}\in \Pi$ is clear. We know that $F\geq 0$ because it is positive semidefinite. Moreover, it cannot vanish on $\R^{n+1}\setminus(0,\dots,0)$ because
$Y(\R)\neq \varnothing$ by Lemma \ref{rev}. This shows that $F\in \Pi$.
\end{proof}

Now choose a general affine subspace $W\subset V$ of dimension $2$ that contains $F$ and $F^{\dagger}$.

\begin{lem}\label{finite}
The set $\Pi\cap W(\R)\cap\Delta(\R)$ is finite.
\end{lem}

\begin{proof}
Let $H\in \Pi\cap W(\R)\cap\Delta(\R)$. Since $H\in \Delta(\R)$, $\{H=0\}\subset\mathbb{P}^n_{\R}$ is a singular hypersurface. Since $H\in \Pi$, $\{H=0\}$ has no real point. Consequently, $\{H=0\}$ has (geometrically) at least two singular points: any singular point and its distinct complex conjugate. This shows that 
$\Pi\cap W(\R)\cap\Delta(\R)\subset W(\R)\cap\Delta'(\R)$. But if $W$ has been chosen to intersect properly $\Delta'$, the variety $W\cap \Delta'$ is already finite.
\end{proof}

\begin{lem}\label{connected} There exists a variety $S$ over $\R$, two points $s, s^{\dagger}\in S(\R)$, and a morphism $\rho : S\to W\setminus\Delta$ such that $S(\R)$ is semi-algebraically connected, $\rho(s)=F$, $\rho(s^{\dagger})=F^{\dagger}$ and $\rho(S(\R))\subset \Pi$.
\end{lem}

\begin{proof}
Choose a coordinate system on $W$ for which $F$ has coordinate $(-1,0)$ and $F^{\dagger}$ has coordinate $(1,0)$. By Lemma \ref{open}, the segment $[F,F^{\dagger}]$ is included in $\Pi$ and $W(\R)\setminus\Pi$ is closed and semi-algebraic.
Consequently, combining \cite[Proposition 2.2.8 (ii)]{BCR} and \cite[Theorem 2.5.8]{BCR}, we see that the distance between $[F,F^{\dagger}]$ and $W(\R)\setminus \Pi$ is positive. It follows that if $\varepsilon\in\R$ is small enough, the ellipse $\{x^2+\frac{1}{\varepsilon}y^2\leq 1\}\subset W(\R)$, that contains $F$ and $F^{\dagger}$, is included in $\Pi$.

Now, consider the double cover $\rho:W':=\{x^2+\frac{1}{\varepsilon}y^2+z^2=1\}\to W$ and define $S:=\rho^{-1}(W\setminus\Delta)\subset W'$. That $\rho(S(\R))$ is included in $\Pi$ and contains $F$ and $F^{\dagger}$ follows from our choice of the ellipse. The semi-algebraic set $W'(\R)$ is a sphere $\mathbb{S}^2$, and $S(\R)$ is the complement of a finite number of points in it by Lemma \ref{finite}. This allows to show by hand that it is semi-algebraically path-connected, hence semi-algebraically connected by \cite[Proposition 2.5.13]{BCR}. 
\end{proof}

Over the base $S$, there is a smooth projective family $\mathcal{Y}\stackrel{p}{\longrightarrow}S$ obtained by pulling back by 
$\rho$ the universal family of smooth double covers over $W\setminus \Delta$. In particular, $\mathcal{Y}_s\simeq Y$ and $\mathcal{Y}_{s^{\dagger}}\simeq Y^{\dagger}$. 
Since $\rho(S(\R))\subset \Pi$, we see that $\mathcal{Y}(\R)=\varnothing$.

\subsection{Cohomology over $\R$ when $d\equiv 0[4]$}
We start with a general lemma:

\begin{lem}\label{cohotop} Let $X$ be a smooth projective geometrically integral variety of dimension $n$ over $\R$ such that $X(\R)=\varnothing$.
Then:
\begin{enumerate}[(i)]
\item $H^{2n}(X,\mathbb{Z}_2(n))\simeq\mathbb{Z}_2$.
\item $H^{2n}(X,\mathbb{Z}_2(n+1))\simeq\mathbb{Z}/2\mathbb{Z}$.
\end{enumerate}
\end{lem}

\begin{proof}
We use the exact sequence (\ref{rc}), as well as Proposition \ref{cd} (ii).

Consider $H^{2n}(X,\mathbb{Z}_2(n))\to H^{2n}(X_{\C},\mathbb{Z}_2)\stackrel{\pi_*}{\longrightarrow}H^{2n}(X,\mathbb{Z}_2(n+1))\to 0$. The cohomology class of a closed point in $H^{2n}(X,\mathbb{Z}_2(n))$ pulls back to twice the cohomology class of a closed point in $H^{2n}(X_{\C},\mathbb{Z}_2)$.  This shows that $H^{2n}(X,\mathbb{Z}_2(n+1))$ is torsion.
From $H^{2n}(X,\mathbb{Z}_2(n+1))\to H^{2n}(X_{\C},\mathbb{Z}_2)\to H^{2n}(X,\mathbb{Z}_2(n))\to 0$, we deduce that $\mathbb{Z}_2\simeq H^{2n}(X_{\C},\mathbb{Z}_2)\to H^{2n}(X,\mathbb{Z}_2(n))$ is an isomorphism. The composition $H^{2n}(X,\mathbb{Z}_2(n))\stackrel{\pi^*}{\longrightarrow} H^{2n}(X_{\C},\mathbb{Z}_2)\stackrel{\pi_*}{\longrightarrow}  H^{2n}(X,\mathbb{Z}_2(n))$ being multiplication by $2$, we see that the image of $H^{2n}(X,\mathbb{Z}_2(n))\stackrel{\pi^*}{\longrightarrow} H^{2n}(X_{\C},\mathbb{Z}_2)$ has index $2$, so that $H^{2n}(X,\mathbb{Z}_2(n+1))=\mathbb{Z}/2\mathbb{Z}$.
\end{proof}


We need information about $\omega^{2n}\in H^{2n}(Y,\mathbb{Z}_2(2n))$, that will be provided by Lemma \ref{vanisheven1} below. 
As a first step towards this result, we deal with the Fermat double cover $Y^{\dagger}:=\{Z^2+F^{\dagger}=0\}$, where $F^{\dagger}:=X_0^d+\dots+X_n^d$.

\begin{lem}\label{vanishFermat}
Suppose that $n$ is odd and $d\equiv 0[4]$. 
Then $\omega^{2n}\in H^{2n}(Y^{\dagger},\mathbb{Z}_2(2n))$ is zero.
\end{lem}

\begin{proof}
The morphism $\mu:Y^{\dagger}\to Q^n$ to $Q^n:=\{Z^2+T_0^2+\dots+T_n^2=0\}\subset\mathbb{P}^{n+1}_{\R}$ defined by $T_i=X_i^{d/2}$ has even degree because $d\equiv 0[4]$. By Lemma \ref{cohotop} applied to $Y^{\dagger}$ and $Q^n$, there is a commutative diagram with surjective vertical arrows:
\begin{equation*}
\xymatrix{
\mathbb{Z}_2=H^{2n}(Q^n_{\C},\mathbb{Z}_2)\ar[d]\ar[r]^-{\mu_{\C}^*}
&  H^{2n}(Y^{\dagger}_{\C},\mathbb{Z}_2)=\mathbb{Z}_2\ar[d]  \\
\mathbb{Z}/2\mathbb{Z}=H^{2n}(Q^n,\mathbb{Z}_2(2n))\ar[r]^-{\mu^*}
& H^{2n}(Y^{\dagger},\mathbb{Z}_2(2n))=\mathbb{Z}/2\mathbb{Z}.
}
\end{equation*}
Since $\mu_{\C}^*$ is the multiplication by the even number $\deg(\mu)$, $\mu^*$ vanishes. Hence so does the composite
$H^{2n}(\R,\mathbb{Z}_2(2n))\to H^{2n}(Q^n,\mathbb{Z}_2(2n))\stackrel{\mu^*}{\longrightarrow}  H^{2n}(Y^{\dagger},\mathbb{Z}_2(2n))$.
\end{proof}

We deduce the same result for $Y$ using a deformation argument: 

\begin{lem}\label{vanisheven1}
Suppose that $n$ is odd and $d\equiv 0[4]$.
Then $\omega^{2n}\in H^{2n}(Y,\mathbb{Z}_2(2n))$ is zero.
\end{lem}

\begin{proof}
Lemma \ref{vanishFermat} and the diagram:
\begin{equation*}
\xymatrix{
H^{2n}(\R,\mathbb{Z}_2(2n)) \ar^{\cong}[r]\ar[d]
& H^{2n}(\R,\mathbb{Z}/2\mathbb{Z})\ar[d]  \\
H^{2n}(Y^{\dagger},\mathbb{Z}_2(2n)) \ar[r]
&  H^{2n}(Y^{\dagger},\mathbb{Z}/2\mathbb{Z})
}
\end{equation*}
shows that $H^{2n}(\R,\mathbb{Z}/2\mathbb{Z})\to H^{2n}(Y^{\dagger},\mathbb{Z}/2\mathbb{Z})$ vanishes.

Now consider the family $\mathcal{Y}\stackrel{p}{\longrightarrow}S$ constructed at the end of paragraph \ref{parsemialg}. The varieties $Y$ and $Y^{\dagger}$ are members of this family and $S(\R)$ is semi-algebraically connected.
If we were working over the field $\mathbb{R}$ of real numbers, we would use topological arguments (namely a $G$-equivariant version of Ehresmann's theorem applied to the fibration $p_{\mathbb{C}}^{-1}(S(\mathbb{R}))\to S(\mathbb{R})$) to show that $H^{2n}(\mathbb{R},\mathbb{Z}/2\mathbb{Z})\to H^{2n}(Y,\mathbb{Z}/2\mathbb{Z})$ vanishes as well. Over an arbitrary real closed field $\R$, the corresponding tools have been developped by Scheiderer \cite{Scheiderer} and the topological arguments may be replaced by \cite[Corollary 17.21]{Scheiderer}. 

Let us explain more precisely how to apply this result. In doing so, we use freely the notations of \cite{Scheiderer}. Consider the composition:
\begin{alignat*}{2}
H^{2n}(\R, &\mathbb{Z}/2\mathbb{Z})\to H^{2n}(\mathcal{Y}, \mathbb{Z}/2\mathbb{Z})\xleftarrow{\sim}H^{2n}(\mathcal{Y}_b,\mathbb{Z}/2\mathbb{Z})\\
&\to H^0(S_b,R^{2n}p_{b*}\mathbb{Z}/2\mathbb{Z})\to H^0(S_r,i^*R^{2n}p_{b*}\mathbb{Z}/2\mathbb{Z}), 
\end{alignat*}
where the isomorphism $H^{2n}(\mathcal{Y}_b,\mathbb{Z}/2\mathbb{Z})\xrightarrow{\sim}H^{2n}(\mathcal{Y}, \mathbb{Z}/2\mathbb{Z})$ follows from \cite[Example 2.14]{Scheiderer}, taking into account Proposition \ref{realff} and the fact that $\mathcal{Y}(\R)=\varnothing$. By proper base change \cite[Theorem 16.2 b)]{Scheiderer} and comparing \'etale and $b$-cohomology using \cite[Example 2.14]{Scheiderer} once again, we see that the stalk of $i^*R^{2n}p_{b*}\mathbb{Z}/2\mathbb{Z}$ at $s$ (resp. $s^{\dagger}$) is $H^{2n}(Y,\mathbb{Z}/2\mathbb{Z})$ (resp. $H^{2n}(Y^{\dagger},\mathbb{Z}/2\mathbb{Z})$). We have proven above that $H^{2n}(\R,\mathbb{Z}/2\mathbb{Z})$ vanishes in the stalk at $s^{\dagger}$. But we know that the sheaf  $i^*R^{2n}p_{b*}\mathbb{Z}/2\mathbb{Z}$ is locally constant on $S_r$ by \cite[Corollary 17.20 b)]{Scheiderer}, and that $S_r$ is connected by \cite[Proposition 7.5.1 (i)]{BCR} and because $S(\R)$ is semi-algebraically connected. Consequently, $H^{2n}(\R,\mathbb{Z}/2\mathbb{Z})$ also vanishes in the stalk at $s$, so that $H^{2n}(\R,\mathbb{Z}/2\mathbb{Z})\to H^{2n}(Y,\mathbb{Z}/2\mathbb{Z})$ is zero.

To conclude that $\omega^{2n}\in  H^{2n}(Y,\mathbb{Z}_2(2n))$ vanishes, consider the exact diagram:
\begin{equation*}
\xymatrix{
&H^{2n}(\R,\mathbb{Z}_2(2n)) \ar^{\cong}[r]\ar[d]
& H^{2n}(\R,\mathbb{Z}/2\mathbb{Z})\ar[d]  \\
H^{2n}(Y,\mathbb{Z}_2(2n)) \ar[r]^{2}&H^{2n}(Y,\mathbb{Z}_2(2n)) \ar[r]
&  H^{2n}(Y,\mathbb{Z}/2\mathbb{Z}),
}
\end{equation*}
and notice that the multiplication by $2$ map is zero by Lemma \ref{cohotop}.
\end{proof}

\begin{prop}\label{vanisheven2}
Suppose that $d\equiv 0[4]$.  Then $\omega^{n+1}\in H^{n+1}(Y,\mathbb{Z}_2(n+1))$ is zero.
\end{prop}

\begin{proof}
Suppose not, and let $k\geq n+1$ be such that $\omega^k\in H^k(Y,\mathbb{Z}_2(k))$ is non-zero and $\omega^{k+1}\in H^{k+1}(Y,\mathbb{Z}_2(k+1))$ vanishes. By Proposition \ref{cd} (ii), $k$ exists and $k\leq 2n$.

Consider the short exact sequence (\ref{rc}) applied to $Y$:
\begin{equation*}
H^k(Y,\mathbb{Z}_2(k+1))\stackrel{\pi^*}{\longrightarrow}
H^k(Y_{\C},\mathbb{Z}_2)\stackrel{\pi_*}{\longrightarrow} H^k(Y,\mathbb{Z}_2(k))\stackrel{\omega}{\longrightarrow}H^{k+1}(Y,\mathbb{Z}_2(k+1)).
\end{equation*}
By hypothesis, $\omega^k\in \Ima(\pi_*)$. By Proposition \ref{cohogeo} (ii), since $\omega^k\in H^k(Y,\mathbb{Z}_2(k))$ is non-zero, $k$ has to be even: $k=2l$.

 If $l$ were even, we would have $H^k(Y_{\C},\mathbb{Z}_2(k+1))^G=0$ by Proposition \ref{cohogeo} (iv), and
the Hochschild-Serre spectral sequence (\ref{HS}) would show that $\pi^*:H^k(Y,\mathbb{Z}_2(k+1))\to H^k(Y_{\C},\mathbb{Z}_2)$ is zero. Consequently, $\Ima(\pi_*)$ has no torsion by Proposition \ref{cohogeo} (iv). This is a contradiction and shows that $l$ is odd. 

Let $H\in H^2(Y,\mathbb{Z}_2(1))$ be the class of $\mathcal{O}_{\mathbb{P}^n_{\R}}(1)$. Since $\pi^*H^l=H_{\C}^l$, and taking into account Proposition \ref{cohogeo} (iv), the only class in $\Ima(\pi_*)$ that may be non-zero is $\pi_*\alpha_l$. Consequently, we have $\omega^{k}=\pi_*\alpha_l$.

Choose by Bertini theorem an $l$-dimensional linear subspace $\mathbb{P}^l_{\R}\subset\mathbb{P}^n_{\R}$ that is transverse to the smooth hypersurface $\{F=0\}$, and define $i:Z\hookrightarrow Y$ to be the inverse image of $\mathbb{P}^l_{\R}$ in $Y$: it is a smooth double cover of $\mathbb{P}^l_{\R}$ ramified over $\{F=0\}\cap \mathbb{P}^l_{\R}$.
Consider the following commutative diagram, where the left horizontal arrows are restrictions, the right horizontal arrows are Gysin morphisms and the vertical ones are those appearing in the exact sequence (\ref{rc}):
\begin{equation*}
\xymatrix{
H^{k}(Y_{\C},\mathbb{Z}_2)\ar[r]^-{i_{\C}^*}\ar[d]^{\pi_*}& H^{k}(Z_{\C},\mathbb{Z}_2)\ar[r]\ar[r]^-{i_{\C*}}\ar[d]^{\pi_*}&H^{2n}(Y_{\C},\mathbb{Z}_2) \ar[d]^{\pi_*}  \\
H^{k}(Y,\mathbb{Z}_2(k))\ar[r]^-{i^*}& H^{k}(Z,\mathbb{Z}_2(k))\ar[r]^-{i_*}&H^{2n}(Y,\mathbb{Z}_2(n+l))   \\
}
\end{equation*}

Look at $\alpha_l\in H^{k}(Y_{\C},\mathbb{Z}_2)$. We have $i_*i^*\pi_*\alpha_l=i_*i^*\omega^{k}=i_*\omega^k=0$ by Lemma \ref{vanisheven1} applied to $Z$. On the other hand, $\pi_*i_{\C*}i_{\C}^*\alpha_l=\pi_*(\alpha_l\cdot[Z_{\C}])=\pi_*(\alpha_l\cdot H_{\C}^{n-l})=\pi_*\alpha_n\neq 0$ because it is the generator of $H^{2n}(Y,\mathbb{Z}_2(n+l))\simeq\mathbb{Z}/2\mathbb{Z}$ as seen in Lemma \ref{cohotop}. This is a contradiction.
\end{proof}

\subsection{Cohomology over $\R$ when $n$ is odd}\label{calculcohod2n}
 As in the previous paragraph, we will reduce the computations we need to the case of a quadric. To do this, we collect a few results about the cohomology of quadrics. Analogues with $2$-torsion coefficients of some of these computations appear in \cite[\S 4]{KS3}.

  We denote by $Q^n:=\{Z^2+T_0^2+\dots+T_n^2=0\}\subset\mathbb{P}^{n+1}_{\R}$ the $n$-dimensional projective anisotropic quadric, by $U^n:=Q^n\setminus Q^{n-1}$ its affine counterpart, and by $H\in H^2(Q^n,\mathbb{Z}_2(1))$ the class of $\mathcal{O}_{\mathbb{P}^{n+1}_{\R}}(1)$.

\begin{lem}\label{cohoUn}The cohomology groups of $U^n$ are as follows:
\begin{equation*}
  H^k(U^n,\mathbb{Z}_2(j))=\left\{
    \begin{array}{@{} l c @{}}
      \mathbb{Z}_2 & \text{if $k=0$ and $j\equiv 0[2]$,} \\
      \mathbb{Z}/2\mathbb{Z}\cdot\omega^k & \text{if $1\leq k\leq n$ and $j\equiv k[2]$,}\\
      \mathbb{Z}_2 & \text{if $k=n$ and $j\equiv n+1[2]$,} \\
      0 & \text{otherwise.}
    \end{array}\right.
\end{equation*}
\end{lem}

\begin{proof}
The geometric cohomology groups of $U^n$ are easily computed as $G$-modules from the description of the geometric cohomology groups of $Q^n$ and $Q^{n-1}$ as $G$-modules given in \cite[XII Th\'eor\`eme 3.3]{SGA72}, and from the long exact sequence of cohomology with support associated with $Q^{n-1}_{\C}\subset Q^n_{\C}$:
$$\dots \to H^{k-2}(Q^{n-1}_{\C},\mathbb{Z}_2(-1))\to H^{k}(Q^n_{\C},\mathbb{Z}_2)\to H^{k}(U^n_{\C},\mathbb{Z}_2)\to\dots$$
One gets:
\begin{equation*}
  H^k(U^n_{\C},\mathbb{Z}_2)=\left\{
    \begin{array}{@{} l c @{}}
      \mathbb{Z}_2 & \text{if $k=0$,} \\
      \mathbb{Z}_2(n+1) & \text{if $k=n$,}\\
      0 & \text{otherwise.}
    \end{array}\right.
  \label{eq4}
\end{equation*}
Consider the Hochschild-Serre spectral sequence (\ref{HS}) for $U^n$. The only possibly non-zero arrows in this spectral sequence are the $d_n:E_{n+1}^{p,n}\to E_{n+1}^{p+n+1,0}$. But these are necessarily surjective because $H^k(U^n,\mathbb{Z}_2(j))=0$ for $k>n$ by Proposition \ref{cd} (iii). This allows to compute the spectral sequence entirely, and to deduce the lemma.
\end{proof}

\begin{lem}\label{generateurs}
Suppose that $n$ is odd. 
Then $H^{n}(Q^{n},\mathbb{Z}_2(n))\simeq (\mathbb{Z}/2\mathbb{Z})^{\lceil \frac{n}{4}\rceil}$ gene\-rated by $\omega^{n},\omega^{n-4}H^2,\dots, \omega^{n-4\lfloor \frac{n}{4}\rfloor}H^{2\lfloor \frac{n}{4}\rfloor}$. 

Moreover, the $2$-torsion subgroup of $H^{n+1}(Q^n,\mathbb{Z}_2(n+1))$ is $H^{n+1}(Q^n,\mathbb{Z}_2(n+1))[2]\simeq (\mathbb{Z}/2\mathbb{Z})^{\lceil \frac{n}{4}\rceil}$, generated by $\omega^{n+1},\omega^{n-3}H^2,\dots, \omega^{n+1-4\lfloor \frac{n}{4}\rfloor}H^{2\lfloor \frac{n}{4}\rfloor}$.
\end{lem}

\begin{proof}
Fix $r\geq 0$. The Gysin morphism $H^{n-2r-2}(Q^{n-r-1},\mathbb{Z}_2(n-r-1))\to H^{n-2r}(Q^{n-r},\mathbb{Z}_2(n-r))$ is part of a long exact sequence of cohomology with supports. In this exact sequence, the morphisms $H^{n-2r-1}(Q^{n-r},\mathbb{Z}_2(n-r))\to H^{n-2r-1}(U^{n-r},\mathbb{Z}_2(n-r))$ and $H^{n-2r}(Q^{n-r},\mathbb{Z}_2(n-r))\to H^{n-2r}(U^{n-r},\mathbb{Z}_2(n-r))$ are surjective. Indeed, in the degrees that come up, all the cohomology of $U^{n-r}$ comes from the base field by Lemma \ref{cohoUn}, hence a fortiori from $Q^{n-r}$.

It follows that this Gysin morphism is injective, with a cokernel naturally isomorphic to $H^{n-2r}(U^{n-r},\mathbb{Z}_2(n-r))=H^{n-2r}(\R,\mathbb{Z}_2(n-r))$. This allows to compute $H^{n-2r}(Q^{n-r},\mathbb{Z}_2(n-r))$ by decreasing induction on $r$, and shows, since $n$ is odd, that $H^{n}(Q^{n},\mathbb{Z}_2(n))\simeq (\mathbb{Z}/2\mathbb{Z})^{\lceil \frac{n}{4}\rceil}$ generated by $\omega^{n},\omega^{n-4}H^2,\dots, \omega^{n-4\lfloor \frac{n}{4}\rfloor}H^{2\lfloor \frac{n}{4}\rfloor}$.

The lemma follows by considering the long exact sequence (\ref{rc}), and using our knowledge of the geometric cohomology of $Q^n$ to get:
\begin{equation}\label{finalomega}0
\to H^n(Q^n,\mathbb{Z}_2(n))\stackrel{\omega}{\longrightarrow} H^{n+1}(Q^n,\mathbb{Z}_2(n+1))\to
 H^{n+1}(Q^n_{\C},\mathbb{Z}_2)\simeq 
\mathbb{Z}_2.\qedhere
\end{equation}
\end{proof}

In the remaining of this paragraph, we will continue to suppose that $n$ is odd. We consider the generator $\gamma$ of $H^{n+1}(Q^n_{\C},\mathbb{Z}_2)\simeq 
\mathbb{Z}_2$ such that $2\gamma=H_{\C}^{\frac{n+1}{2}}$ \cite[XII Th\'eor\`eme 3.3]{SGA72}.

If $n\equiv 1[4]$, define $\delta:=\pi_*\gamma\in H^{n+1}(Q^n,\mathbb{Z}_2(n+1))$. If $n\equiv 3[4]$, define $\delta:=\pi_*\gamma-H^{\frac{n+1}{2}}\in H^{n+1}(Q^n,\mathbb{Z}_2(n+1))$.

\begin{lem}\label{pasvanishoddquadric} Suppose that $n$ is odd. Then $\delta\in H^{n+1}(Q^n,\mathbb{Z}_2(n+1))[2]$. Moreover, $\delta$ does not belong to the subgroup of $H^{n+1}(Q^n,\mathbb{Z}_2(n+1))[2]$ generated by $\omega^{n-3}H^2,\dots, \omega^{n+1-4\lfloor \frac{n}{4}\rfloor}H^{2\lfloor \frac{n}{4}\rfloor}$.
\end{lem}

\begin{proof}
We suppose that $n\equiv 1[4]$, the arguments when $n\equiv 3[4]$ are analogous.

Since $G$ acts on $H^{n+1}(Q^n_{\C},\mathbb{Z}_2)\simeq 
\mathbb{Z}_2$ by multiplication by $(-1)^{\frac{n+1}{2}}$ and since $n+1\not\equiv\frac{n+1}{2}[2]$, the natural morphism $H^{n+1}(Q^n,\mathbb{Z}_2(n+1))\to H^{n+1}(Q^n_{\C},\mathbb{Z}_2)$ is zero. From the exact sequence (\ref{finalomega}), this implies that $H^{n+1}(Q^n,\mathbb{Z}_2(n+1))$ is a $2$-torsion group, so that $\delta$ is $2$-torsion.

Let us check that $\delta$ is nonzero. From the exact sequence:
$$H^{n+1}(Q^n,\mathbb{Z}_2(n))\stackrel{\pi^*}{\longrightarrow}  H^{n+1}(Q^n_{\C}, \mathbb{Z}_2)\stackrel{\pi_*}{\longrightarrow} H^{n+1}(Q^n,\mathbb{Z}_2(n+1)),$$
we see that it suffices to prove that there does not exist a cohomology class $\zeta\in H^{n+1}(Q^n,\mathbb{Z}_2(n))$ such that $\pi^*\zeta=\gamma$. If such a class existed, $\pi^*(\zeta\cdot H^{\frac{n-1}{2}})$ would be a generator of $H^{2n}(Q^n_{\C},\mathbb{Z}_2)$. This would contradict the fact, proven in Lemma \ref{cohotop}, that the cokernel of $H^{2n}(Q^n,\mathbb{Z}_2(n))\to H^{2n}(Q^n_{\C},\mathbb{Z}_2)$ is $H^{2n}(Q^n,\mathbb{Z}_2(n+1))\simeq \mathbb{Z}/2\mathbb{Z}$.      

 Introduce now the commutative diagram below, whose horizontal arrows are Gysin morphisms:
\begin{equation}\label{Gysins}
\begin{gathered}
\xymatrix@C=1em{
H^{n-3}(Q^{n-2},\mathbb{Z}_2(n-1))\ar[r]\ar[d]^{\omega}& H^{n-1}(Q^{n-1},\mathbb{Z}_2(n))\ar[r]\ar[d]^{\omega}&H^{n+1}(Q^{n},\mathbb{Z}_2(n+1)) \ar[d]^{\omega}  \\
H^{n-2}(Q^{n-2},\mathbb{Z}_2(n))\ar[r]& H^{n}(Q^{n-1},\mathbb{Z}_2(n+1))\ar[r]&H^{n+2}(Q^{n},\mathbb{Z}_2(n+2))  \\
}
\end{gathered}
\end{equation}

Let us show that the lower horizontal arrows of (\ref{Gysins}) are injective.
Considering $\omega^{n-1}\in H^{n-1}(Q^{n-1},\mathbb{Z}_2(n+1))$ and using Lemma \ref{cohoUn} shows that $H^{n-1}(Q^{n-1},\mathbb{Z}_2(n+1))\to H^{n-1}(U^{n-1},\mathbb{Z}_2(n+1))$ is surjective, so that the map $H^{n-2}(Q^{n-2},\mathbb{Z}_2(n))\to H^{n}(Q^{n-1},\mathbb{Z}_2(n+1))$ is injective. Since $H^{n+1}(U^{n},\mathbb{Z}_2(n+2))=0$ by Lemma \ref{cohoUn}, $H^{n}(Q^{n-1},\mathbb{Z}_2(n+1))\to H^{n+2}(Q^{n},\mathbb{Z}_2(n+2))$ is also injective.

Suppose for contradiction that $\delta$ may be written as a linear combination of the classes $\omega^{n-3}H^2,\dots, \omega^{n+1-4\lfloor \frac{n}{4}\rfloor}H^{2\lfloor \frac{n}{4}\rfloor}$. Then, it is the image by the Gysin morphism $H^{n-3}(Q^{n-2},\mathbb{Z}_2(n-1))\to H^{n+1}(Q^{n},\mathbb{Z}_2(n+1))$ of a class $\epsilon$ that is a linear combination of $\omega^{n-3},\dots, \omega^{n+1-4\lfloor \frac{n}{4}\rfloor}H^{2\lfloor \frac{n}{4}\rfloor-2}$. The image of $\epsilon $ in $H^{n+2}(Q^{n},\mathbb{Z}_2(n+2))$ is $\delta\cdot\omega=\pi_*\gamma\cdot\omega=0$. By the injectivity result just proved, $\epsilon\cdot\omega=0$.
It follows that $\omega^{n-2},\dots, \omega^{n+2-4\lfloor \frac{n}{4}\rfloor}H^{2\lfloor \frac{n}{4}\rfloor-2}$ are not independent in $H^{n-2}(Q^{n-2},\mathbb{Z}_2(n))$, contradicting Lemma \ref{generateurs}.
\end{proof}


We return to our double cover $Y\to\mathbb{P}^n_{\R}$. We recall from Proposition \ref{cohogeo} that $H^{n+1}(Y_{\C},\mathbb{Z}_2)=\mathbb{Z}_2(n+1)$ with a generator $\alpha:=\alpha_{\frac{n+1}{2}}$ such that $2\alpha=H_{\C}^{\frac{n+1}{2}}$.

If $n\equiv 1[4]$, define $\beta:=\pi_*\alpha\in H^{n+1}(Y,\mathbb{Z}_2(n+1))$. If $n\equiv 3[4]$, define $\beta:=\pi_*\alpha-H^{\frac{n+1}{2}}\in H^{n+1}(Y,\mathbb{Z}_2(n+1))$, where $H\in H^2(Y,\mathbb{Z}_2(1))$ is the class of $\mathcal{O}_{\mathbb{P}^n_{\R}}(1)$.

\begin{prop}\label{combiodd} 
Suppose that $n$ is odd. Then $\omega^{n+1}$ is a linear combination of $\omega^{n-3}H^2,\dots, \omega^{n+1-4\lfloor \frac{n}{4}\rfloor}H^{2\lfloor \frac{n}{4}\rfloor}$ and $\beta$ in $H^{n+1}(Y,\mathbb{Z}_2(n+1))$.
\end{prop}

\begin{proof}
Using a deformation argument as in the proof of Lemma \ref{vanisheven1}, one reduces to the case of the Fermat double cover $Y^{\dagger}:=\{Z^2+F^{\dagger}=0\}$, where $F^{\dagger}:=X_0^d+\dots+X_n^d$.

As in Lemma \ref{vanishFermat}, we consider the morphism $\mu:Y^{\dagger}\to Q^n$ to the quadric $Q^n:=\{Z^2+T_0^2+\dots+T_n^2=0\}\subset\mathbb{P}^{n+1}_{\R}$, defined by $T_i=X_i^{d/2}$. Note that $\mu^*H=(\frac{d}{2})H$, so that $\mu_{\C}^*\gamma=(\frac{d}{2})^{\frac{n+1}{2}}\alpha$ and $\mu^*\delta=(\frac{d}{2})^{\frac{n+1}{2}}\beta$.

By Lemma \ref{generateurs}, the group $H^{n+1}(Q^n,\mathbb{Z}_2(n+1))[2]$ is freely generated by $\omega^{n+1}$, $\omega^{n-3}H^2, \dots,  \omega^{n+1-4\lfloor \frac{n}{4}\rfloor}H^{2\lfloor \frac{n}{4}\rfloor}$. By Lemma \ref{pasvanishoddquadric}, $\delta$ does not belong to the subgroup generated by $\omega^{n-3}H^2, \dots, \omega^{n+1-4\lfloor \frac{n}{4}\rfloor}H^{2\lfloor \frac{n}{4}\rfloor}$. Consequently, $\omega^{n+1}$ is a linear combination of $\delta,\omega^{n-3}H^2,\dots, \omega^{n+1-4\lfloor \frac{n}{4}\rfloor}H^{2\lfloor \frac{n}{4}\rfloor}$
in $H^{n+1}(Q^n,\mathbb{Z}_2(n+1))[2]$.

Pulling back this relation by $\mu$, it follows that $\omega^{n+1}$ is a linear combination of $\beta,\omega^{n-3}H^2,\dots, \omega^{n+1-4\lfloor \frac{n}{4}\rfloor}H^{2\lfloor \frac{n}{4}\rfloor}$
in $H^{n+1}(Y^{\dagger},\mathbb{Z}_2(n+1))$, as wanted.
\end{proof}

\begin{cor}\label{coniveaulien}
Suppose that $n$ is odd.
If $\alpha\in H^{n+1}(Y_{\C},\mathbb{Z}_2)$ has coniveau $\geq 2$, then $\omega^{n+1}\in H^{n+1}(Y,\mathbb{Z}_2(n+1))$ has coniveau $\geq 2$.
\end{cor}

\begin{proof}
The statement follows from Proposition \ref{combiodd}, because the cohomology classes $\omega^{n-3}H^2, \dots, \omega^{n+1-4\lfloor \frac{n}{4}\rfloor}H^{2\lfloor \frac{n}{4}\rfloor}$, as well as $H^{\frac{n+1}{2}}$ when $n\neq 1$, obviously have coniveau $\geq 2$ as multiples of $H^2$. 
\end{proof}

\section{A geometric coniveau computation}\label{s5}

In this section, we fix $n\geq 3$ an odd integer and we take $d:=2n$. 

We work over an algebraically closed field $\C$ of characteristic $0$.
We consider $F\in\C[X_0,\dots,X_n]_d$ a homogeneous degree $d$ polynomial such that $\{F=0\}$ is smooth. Let $Y$ be the double cover of $\mathbb{P}^n_{\C}$ ramified over $\{F=0\}$ defined by the equation $Y:=\{Z^2+F=0\}$, and $H:=\mathcal{O}_{\mathbb{P}^n_{\C}}(1)$.

By Proposition \ref{cohogeo}, $H^{n+1}(Y,\mathbb{Z}_2)$ has a generator $\alpha$ such that $2\alpha=H^{\frac{n+1}{2}}$.
In order to apply Corollary \ref{coniveaulien}, we need to answer positively the following:

\begin{quest}\label{qconiveau}
Does $\alpha$ have coniveau $\geq 2$ ?
\end{quest}

When $n=3$, $Y$ is a sextic double solid, and this is very easy: 

\begin{lem}\label{coniveau3}
If $n=3$, $\alpha$ has coniveau $\geq 2$.
\end{lem}

\begin{proof}
A dimension count (see for instance Lemma \ref{varlines} below) shows that $Y$ contains a line, that is a curve of degree $1$ against $H$. The cohomology class of such a curve is $\alpha$, so that $\alpha$ is algebraic, hence of coniveau $2$.
\end{proof}

In what follows, we answer positively Question \ref{qconiveau} when $n=5$, following an argument of Voisin. We comment on the $n\geq 7$ case in paragraph \ref{Remarks}.

\begin{prop}\label{coniveau5}
If $n=5$, $\alpha$ has coniveau $\geq 2$.
\end{prop}


\subsection{Reductions}
The following reductions are standard. We include them because we do not know a convenient reference. 

\begin{lem}\label{reducextensions}
Let $\C\subset\C'$ be an extension of algebraically closed fields.
Then the answer to Question \ref{qconiveau} is positive for $Y$ over $\C$ if and only if it is for $Y_{\C'}$ over $\C'$.
\end{lem}

\begin{proof}
 It is clear that if $\alpha$ has coniveau $\geq 2$ over $\C$, it has coniveau $\geq 2$ over $\C'$. 

Suppose conversely that there is a closed subset $Z\subset Y_{\C'}$ of codimension $\geq 2$ such that $\alpha_{\C'}$ vanishes in $Y_{\C'}\setminus Z$. Taking an extension of finite type of $\C$ over which $Z$ is defined, spreading out and shrinking the base gives a smooth integral variety $B$ over $\C$, and a subvariety $Z_B\subset Y\times B$ of codimension $\geq 2$ in the fibers of $p_2:Y\times B\to B$ such that $p_1^*\alpha$ vanishes in the generic geometric fiber of $p_2:(Y\times B\setminus Z_B)\to B$. The existence of cospecialization maps for smooth morphisms \cite[Arcata V (1.6)]{SGA45} implies that $\alpha$ vanishes in every geometric fiber of $p_2:(Y\times B\setminus Z_B)\to B$. Taking the fiber over a $\C$-point of $B$ shows that $\alpha$ has coniveau $\geq 2$.
\end{proof}

\begin{lem}\label{reducspe}
In order to answer positively Question \ref{qconiveau} in general, it suffices to answer it over the field $\mathbb{C}$ of complex numbers, for a general choice of $F$.
\end{lem}

\begin{proof}
Let $U\subset \mathbb{C}[X_0,\dots,X_n]_d$ be a Zariski-open subset of degree $d$ polynomials $F$ as in the hypothesis: for $F\in U(\mathbb{C})$, $\{F=0\}$ is smooth and $\alpha$ is of coniveau $\geq 2$.

  Let $K$ be an algebraic closure of the function field $\mathbb{C}(U)$, let $F_K$ be the generic polynomial and let $Y_K$ be the associated universal double cover. Choosing an isomorphism $K\simeq\mathbb{C}$ such that the induced polynomial $F_{\mathbb{C}}$ belongs to $U$
shows that the answer to Question \ref{qconiveau} is positive for $Y_K$.

  Let us now deal with the case $\C=\mathbb{C}$. It is possible to find the spectrum $T$ of a strictly henselian discrete valuation ring and a morphism $T\to \mathbb{C}[X_0,\dots,X_n]_d$ sending the closed point of $T$ to the polynomial associated to $Y$ and its generic point to the generic point of $U$. Let $Y_T$ be the induced family of double covers. Up to replacing $T$ by a finite extension, there exists a codimension $2$ subset $Z\subset Y_T$ flat over $T$ such that $\alpha$ vanishes in the complement of $Z$ in the generic geometric fiber. Using cospecialization maps again shows that $\alpha$ vanishes in the complement of $Z$ in the special fiber, so that the answer to Question \ref{qconiveau} is positive for $Y$.

  In general, choose an algebraically closed subfield of finite transcendence degree of $\C$ over which $Y$ is defined, embed it in $\mathbb{C}$, and apply Lemma \ref{reducextensions}.
\end{proof}

\subsection{The variety of lines}
We define $F(Y)$ to be the Fano variety of lines of $Y$, that is the Hilbert scheme of $Y$ parametrizing degree $1$ curves in $Y$. We also introduce the universal family $I\subset Y\times F(Y)$ and denote by $q:I\to Y$ and $p:I\to F(Y)$ the natural projections.

  The following lemma is well-known for hypersurfaces \cite[\S 3]{BVdV} or complete intersections \cite[Th\'eor\`eme 2.1]{DebarreManivel}, and the proof in our situation is similar.

\begin{lem}\label{varlines}
If $Y$ is general, $F(Y)$ is smooth, non-empty and of dimension $n-2$.
\end{lem}

\begin{proof}
  First, an easy dimension count shows that if $F$ is general, $\{F=0\}$ contains no line \cite[\S 3]{BVdV}. It follows that for such a general $Y$, $F(Y)$ is a double \'etale cover of the variety $G(Y)$ of lines in $\mathbb{P}^n_{\C}$ on which the restriction of $F$ is a square. 

  One introduces the space $V\subset \C[X_0,\dots,X_n]_d$ of degree $d$ polynomials $F$ such that $\{F=0\}$ is smooth.
We consider the universal double cover $Y_V\to V$, and the universal Fano variety of lines $F(Y)_V\to V$, viewed as a double
cover of $G(Y)_V$ that is \'etale generically over $V$. Looking at the natural projection from $G(Y)_V$ to the grassmannian of lines in $\mathbb{P}^n_{\C}$, one sees that $G(Y)_V$ is smooth of dimension $\dim(V)+n-2$. Since we are in characteristic $0$, the general fiber of $G(Y)_V\to V$ is smooth and it remains to show that this morphism is dominant, or that $F(Y)_V\to V$ is dominant. We will do it by finding one point at which it is smooth.

  To do so, we fix a line $L$ with equations $X_2=\dots=X_n=0$, and a degree $d$ polynomial $F$ such that the restriction of $F$ to $L$ is the square of $H\in\C[X_0,X_1]_n$. The double cover $Y$ may be viewed naturally as the zero-locus of a section of $\mathcal{O}(d)$ in the total space $E\to \mathbb{P}^n_{\C}$ of the line bundle $\mathcal{O}_{\mathbb{P}^n_{\C}}(n)$.
The inverse image of $L$ in $Y$ splits into the union of two lines: let $\Lambda$ be one of them. The normal exact sequence $0\to N_{\Lambda/Y}\to N_{\Lambda/E}\to N_{Y/E}|_{\Lambda}\to 0$ reads:
$$0\to N_{\Lambda/Y}\to \mathcal{O}(1)^{\oplus n-1}\oplus\mathcal{O}(n)\to \mathcal{O}(2n)\to 0.$$
The same computation as the one carried out in \cite[\S 2]{BVdV} for hypersurfaces shows that the last arrow is given by $(\frac{\partial F}{\partial X_2},\dots,\frac{\partial F}{\partial X_n}, H)$. Consequently, if $H^0(\Lambda, \mathcal{O}(2n))$ is generated by multiples of $\frac{\partial F}{\partial X_2},\dots,\frac{\partial F}{\partial X_n}$ and $H$, $H^1(\Lambda, N_{\Lambda/Y})=0$ and $\Lambda$ corresponds to a smooth point of the relative Hilbert scheme $F(Y)_V\to V$, as wanted. It is easy to find a polynomial $F$ satisfying this condition.
\end{proof}

\begin{lem}\label{incidence}
  If there exists a smooth $Y$ over $\mathbb{C}$ such that $F(Y)$ is smooth of dimension $n-2$, and a cohomology class $\zeta\in H^{n-1}(F(Y),\mathbb{Z}_2)$ such that $q_*p^*\zeta$ is an odd multiple of $\alpha$, then Question \ref{qconiveau} has a positive answer.
\end{lem}

\begin{proof}
  By Lemma \ref{reducspe} and Lemma \ref{varlines}, it suffices to consider a double cover over the complex numbers whose variety of lines is smooth of dimension $n-2$. By Ehresmann's theorem, the existence of a cohomology class $\zeta$ as in our hypothesis does not depend on $Y$ (as long as $Y$ and $F(Y)$ are smooth). Consequently, it suffices to answer Question \ref{qconiveau} for a double cover $Y$ for which such a $\zeta$ exists.

  On the one hand $2\alpha=H^{\frac{n+1}{2}}$ is algebraic, hence of coniveau $\geq 2$. On the other hand $\zeta$ has coniveau $\geq 1$ because it vanishes on any affine open of $F(Y)$. It follows that $q_*p^*\zeta$ has coniveau $\geq 2$ because $q:I\to Y$ is not dominant by dimension. Combining these two assertions, we see that $\alpha$ has coniveau $\geq 2$.
\end{proof}

\subsection{A degeneration argument}\label{degeneration}

In this paragraph, we set $n=5$ and $d=10$, and we prove Proposition \ref{coniveau5} by checking the hypothesis of Lemma \ref{incidence}.

To do so, we choose four homogeneous polynomials $P\in\mathbb{C}[X_0,\dots,X_5]_5$, $G\in\mathbb{C}[X_1,X_2,X_3]_5$,  and $Q_1,Q_2\in \mathbb{C}[X_0,\dots,X_5]_9$, and we set
$$F_0:=X_0^{10}+PG+Q_1X_4+Q_2X_5\in \mathbb{C}[X_0,\dots,X_5]_{10}.$$
The reason for this choice is that $F_0$ restricts to a square on the cone $\Gamma:=\{G=X_4=X_5=0\}$, so that the inverse image of $\Gamma$ in $Y_0:=\{Z^2+F_0=0\}$ has two irreducible components, giving rise to two $1$-dimensional families of lines in $Y_0$. We denote by $\Phi_0\subset F(Y_0)$ the curve corresponding to one of these families: it is naturally isomorphic to $\{G=0\}\subset\mathbb{P}^2_{\mathbb{C}}$.

\begin{lem}\label{equationchoice}
If $P,G,Q_1,Q_2$ have been chosen general, $Y_0$ is smooth, $\Phi_0$ is smooth and $F(Y_0)$ is smooth of dimension $3$ along $\Phi_0$.
\end{lem}

\begin{proof}
  To check that the general zero-locus of such an $F_0$ is smooth, it suffices to deal with equations of the form $\lambda X_0^{10}+\mu X_1^{10}+Q_1X_4+Q_2X_5\in \mathbb{C}[X_0,\dots,X_5]_{10}$. These form a linear system, so a general one among these is smooth outside of the base locus $\{X_0=X_1=X_4=X_5=0\}$ by Bertini theorem. But there exists a particular one that is smooth on the base locus: take $Q_1=X_2^9$ and $Q_2=X_3^9$. It follows that the general one is smooth everywhere.

 To conclude, fix $G$ such that $\Phi_0\simeq\{G=0\}\subset\mathbb{P}^2_{\mathbb{C}}$ is smooth. It suffices to prove that for every $\Lambda\in \Phi_0$, $F(Y_0)$ is smooth of dimension $3$ at $\Lambda$, with possible exceptions on a codimension $2$ subset of the parameter space for $P,Q_1$ and $Q_2$.
By the computations in the proof of Lemma \ref{varlines}, we need to show that, outside such a subset, $H^0(\Lambda,\mathcal{O}(10))$ is generated by multiples of $\frac{\partial F_0}{\partial X_2},\dots,\frac{\partial F_0}{\partial X_5}$ and $X_0^5$. 

This amounts to showing that, outside of a codimension $2$ subset of the parameter space for $P\in\mathbb{C}[X_0,X_1]_5$ and $Q_1,Q_2\in \mathbb{C}[X_0,X_1]_9$, $H^0(\mathbb{P}^1,\mathcal{O}(10))$ is generated by multiples of $X_0^5, PX_1^4,Q_1$ and $Q_2$. This is easy to see, by exhibiting a complete curve in the projectivized parameter space avoiding the bad locus.
\end{proof}


Now, let $\Delta$ be a small enough disk in $\mathbb{C}[X_0,\dots,X_5]_{10}$ centered around the polynomial $F_0$ given by Lemma \ref{equationchoice}. Let $Y_{\Delta}$ be the family of double covers over it, and $F(Y)_{\Delta}$ the corresponding family of varieties of lines. 

Recall that $\Phi_0$ is a smooth proper subvariety of the smooth locus of the special fiber $F(Y_0)$. Using the flow of a vector field as in the
proof of Ehresmann's theorem, one sees that $\Phi_0$ deforms (as a differentiable submanifold) to nearby fibers for which $F(Y_t)$ is smooth,
giving rise to a cohomology class $\zeta_t=[\Phi_t]\in H^{n-1}(F(Y_t),\mathbb{Z}_2)$. We compute $q_*p^*\zeta_t=q_*[p^{-1}(\Phi_t)]=q_*[p^{-1}(\Phi_0)]$: it is the cycle class of the cone $\Gamma$, that is equal to $5\alpha$.

\subsection{Remarks}\label{Remarks}

When $n\geq 7$, an argument analogous to that of paragraph \ref{degeneration} fails, because one gets a double cover $Y_0$ whose variety of lines is singular along a codimension $2$ subset of the subvariety $\Phi_0$ that we would like to deform to nearby fibers $F(Y_t)$.

It might still be possible to show, by another argument, the existence of a cohomology class $\zeta$ allowing to apply Lemma \ref{incidence}. To do so, one would need to compute part of the integral cohomology of $F(Y)$. The rational cohomology of $F(Y)$ in the required degree is well understood thanks to \cite[Th\'eor\`eme 3.4]{DebarreManivel} (where the computations are carried out in the analogous setting of hypersurfaces or complete intersections). However, Debarre and Manivel's approach, relying on the Hodge decomposition, does not allow to control the integral cohomology groups of $F(Y)$.

\section{Proof of the main theorem}\label{s6}

We now come back to our main goal: the proof of Theorem \ref{main}.

\subsection{The generic case}

Let us first put together what we have obtained so far. Fix $n\geq 2$. Define $d(n)$ by setting $d(n):=2n$ if $n$ is even or equal to $3$ or $5$ and $d(n):=2n-2$ if $n\geq 7$ is odd.

\begin{prop}\label{proofgeneric}
Let $f\in \R[X_1,\dots,X_n]$ be a positive semidefinite polynomial of degree $d(n)$ whose homogenization $F$ defines a smooth hypersurface in $\mathbb{P}^n_{\R}$. Then $f$ is a sum of $2^n-1$ squares in $\R(X_1,\dots,X_n)$.
\end{prop}

\begin{proof}
We consider the double cover $Y$ of $\mathbb{P}^n_{\R}$ ramified over $\{F=0\}$ defined by the equation $Y:=\{Z^2+F=0\}$. The variety $Y$ is smooth, $Y(\R)=\varnothing$ by Lemma \ref{rev}, and computing that the anticanonical bundle $-K_Y=\mathcal{O}_{\mathbb{P}^n_{\R}}(n+1-\frac{d(n)}{2})$ of $Y$ is ample, one sees that $Y_{\C}$ is Fano, hence rationally connected. 

By Proposition \ref{sumlevel}, we need to show that the level of $\R(Y)$ is $<2^n$. Applying Proposition \ref{levelnr} (iii)$\Rightarrow$(i), we have to prove that $\omega^n\in H^n(Y,\Z_2(n))$ has coniveau $\geq 1$. Finally, since $Y_{\C}$ is rationally connected, the converse Proposition \ref{nn+1} (ii)$\Rightarrow$(i) holds: we only have to check that $\omega^{n+1}\in H^{n+1}(Y,\mathbb{Z}_2(n+1))$ has coniveau $\geq 2$.

When $n\neq 3$ or $5$, $d(n)\equiv 0[4]$ so that $\omega^{n+1}\in H^{n+1}(Y,\mathbb{Z}_2(n+1))$ vanishes by Proposition \ref{vanisheven2}.

When $n=3$ or $5$, $\omega^{n+1}\in H^{n+1}(Y,\mathbb{Z}_2(n+1))$ is seen to be of coniveau $\geq 2$ by combining Corollary \ref{coniveaulien} and either Lemma \ref{coniveau3} when $n=3$ or Proposition \ref{coniveau5} when $n=5$.
\end{proof}

\subsection{A specialization argument}

  We do not know how to deal with singular equations using the same arguments because one has too little control on the geometry of (a resolution of singularities of) the variety $Y$. Instead, we rely on a specialization argument, that will also take care of the lower values of the degree.

\begin{thm}[Theorem \ref{main}]
Let $f\in \R[X_1,\dots,X_n]$ be a positive semidefinite polynomial of degree $\leq d(n)$. Then $f$ is a sum of $2^n-1$ squares in $\R(X_1,\dots,X_n)$.
\end{thm}

\begin{proof}
Consider $g:=f+t(1+\sum_{i=1}^nX_i^{d(n)})\in \R(t)[X_1,\dots,X_n]$. It is a degree $d(n)$ polynomial whose homogenization defines a smooth hypersurface in $\mathbb{P}^n_{\R}$ because so does its specialization $1+\sum_{i=1}^nX_i^{d(n)}$. Let $\s:=\cup_r\R((t^{1/r}))$ be a real closed extension of $\R(t)$. By Artin's solution to Hilbert's 17\textsuperscript{th} problem \cite{Artin17}, $f$ is a sum of squares in $\R(X_1,\dots,X_n)$, hence still a positive semidefinite polynomial viewed in $\s[X_1,\dots,X_n]$. Consequently, since $t=(t^{1/2})^2$ is a square in $\s$, $g\in \s[X_1,\dots,X_n]$ is a positive semidefinite polynomial. Applying Proposition \ref{proofgeneric} over the real closed field $\s$, we see that $g$ is a sum of $2^n-1$ squares in $\s(X_1,\dots,X_n)$: one has $g=\sum_{i=1}^{2^n-1} h_i^2$.

Consider the $t$-adic valuation on $\s$.
Applying $n$ times successively \cite[Chapitre VI \S10, Proposition 2]{BourbakiCA}, we can extend it to a valuation $v$ on $\s(X_1,\dots,X_n)$ that is trivial on $\R(X_1,\dots,X_n)$, and whose residue field is isomorphic to $\R(X_1,\dots,X_n)$. Note that these choices imply that $v(g)=0$ and that the reduction of $g$ modulo $v$ is $f\in \R(X_1,\dots,X_n)$.

Define $m:=\inf_i v(h_i)$ and notice that $m\leq 0$ because $v(g)=0$. 
Suppose for contradiction that $m<0$ and let $j$ be such that $v(h_j)=m$.
Then it is possible to reduce the equality $gh_j^{-2}=\sum_{i=1}^{2^n-1} (h_ih_j^{-1})^2$ modulo $v$. This is absurd because we get a non-trivial sum of squares that is zero in $\R(X_1,\dots,X_n)$. This shows that $m=0$. Consequently, it is possible to reduce the equality $g=\sum_{i=1}^{2^n-1} h_i^2$ modulo $v$, showing that $f$ is a sum of $2^n-1$ squares in $\R(X_1,\dots,X_n)$ as wanted.
\end{proof}

We conclude by stating explicitely the following consequence of our proof:

\begin{prop}\label{conditional}
If $n\geq 3$ is odd and if Question \ref{qconiveau} has a positive answer, then Theorem \ref{main} also holds in $n$ variables and degree $d=2n$.
\end{prop}



\bibliographystyle{plain}
\bibliography{biblio}

\begin{thebibliography}{10}

\bibitem{SGA42}
{\em Th\'eorie des topos et cohomologie \'etale des sch\'emas. {T}ome 2}.
\newblock Lecture Notes in Mathematics, Vol. 270. Springer-Verlag, Berlin-New
  York, 1972.
\newblock S{\'e}minaire de G{\'e}om{\'e}trie Alg{\'e}brique du Bois-Marie
  1963--1964 (SGA 4), Dirig{\'e} par M. Artin, A. Grothendieck et J. L.
  Verdier. Avec la collaboration de N. Bourbaki, P. Deligne et B. Saint-Donat.

\bibitem{SGA72}
{\em Groupes de monodromie en g\'eom\'etrie alg\'ebrique. {II}}.
\newblock Lecture Notes in Mathematics, Vol. 340. Springer-Verlag, Berlin-New
  York, 1973.
\newblock S{\'e}minaire de G{\'e}om{\'e}trie Alg{\'e}brique du Bois-Marie
  1967--1969 (SGA 7 II), Dirig{\'e} par P. Deligne et N. Katz.

\bibitem{SGA43}
{\em Th\'eorie des topos et cohomologie \'etale des sch\'emas. {T}ome 3}.
\newblock Lecture Notes in Mathematics, Vol. 305. Springer-Verlag, Berlin-New
  York, 1973.
\newblock S{\'e}minaire de G{\'e}om{\'e}trie Alg{\'e}brique du Bois-Marie
  1963--1964 (SGA 4), Dirig{\'e} par M. Artin, A. Grothendieck et J. L.
  Verdier. Avec la collaboration de P. Deligne et B. Saint-Donat.

\bibitem{Artin17}
E.~Artin.
\newblock \"{U}ber die {Z}erlegung definiter {F}unktionen in {Q}uadrate.
\newblock {\em Abh. Math. Sem. Univ. Hamburg}, 5(1):100--115, 1927.

\bibitem{BVdV}
W.~Barth and A.~Van~de Ven.
\newblock Fano varieties of lines on hypersurfaces.
\newblock {\em Arch. Math. (Basel)}, 31(1):96--104, 1978/79.

\bibitem{BO}
S.~Bloch and A.~Ogus.
\newblock Gersten's conjecture and the homology of schemes.
\newblock {\em Ann. Sci. \'Ecole Norm. Sup. (4)}, 7:181--201 (1975), 1974.

\bibitem{BS}
S.~Bloch and V.~Srinivas.
\newblock Remarks on correspondences and algebraic cycles.
\newblock {\em Amer. J. Math.}, 105(5):1235--1253, 1983.

\bibitem{BCR}
J.~Bochnak, M.~Coste, and M.-F. Roy.
\newblock {\em Real algebraic geometry}, volume~36 of {\em Ergebnisse der
  Mathematik und ihrer Grenzgebiete (3)}.
\newblock Springer-Verlag, Berlin, 1998.
\newblock Translated from the 1987 French original, Revised by the authors.

\bibitem{BourbakiCA}
N.~Bourbaki.
\newblock {\em \'{E}l\'ements de math\'ematique}.
\newblock Masson, Paris, 1985.
\newblock Alg{\`e}bre commutative. Chapitres 5 {\`a} 7. Reprint.

\bibitem{CEP}
J.~W.~S. Cassels, W.~J. Ellison, and A.~Pfister.
\newblock On sums of squares and on elliptic curves over function fields.
\newblock {\em J. Number Theory}, 3:125--149, 1971.

\bibitem{Clemens}
C.~H. Clemens.
\newblock Double solids.
\newblock {\em Adv. in Math.}, 47(2):107--230, 1983.

\bibitem{CTrealrat}
J.-L. Colliot-Th{\'e}l{\`e}ne.
\newblock Real rational surfaces without a real point.
\newblock {\em Arch. Math. (Basel)}, 58(4):392--396, 1992.

\bibitem{CTNL}
J.-L. Colliot-Th{\'e}l{\`e}ne.
\newblock The {N}oether-{L}efschetz theorem and sums of {$4$}-squares in the
  rational function field {$\bold R(x,y)$}.
\newblock {\em Compositio Math.}, 86(2):235--243, 1993.

\bibitem{BOG}
J.-L. Colliot-Th{\'e}l{\`e}ne, R.~T. Hoobler, and B.~Kahn.
\newblock The {B}loch-{O}gus-{G}abber theorem.
\newblock In {\em Algebraic {$K$}-theory ({T}oronto, {ON}, 1996)}, volume~16 of
  {\em Fields Inst. Commun.}, pages 31--94. Amer. Math. Soc., Providence, RI,
  1997.

\bibitem{CTParimala}
J.-L. Colliot-Th{\'e}l{\`e}ne and R.~Parimala.
\newblock Real components of algebraic varieties and \'etale cohomology.
\newblock {\em Invent. Math.}, 101(1):81--99, 1990.

\bibitem{CTS}
J.-L. Colliot-Th{\'e}l{\`e}ne and C.~Scheiderer.
\newblock Zero-cycles and cohomology on real algebraic varieties.
\newblock {\em Topology}, 35(2):533--559, 1996.

\bibitem{CTV}
J.-L. Colliot-Th{\'e}l{\`e}ne and C.~Voisin.
\newblock Cohomologie non ramifi\'ee et conjecture de {H}odge enti\`ere.
\newblock {\em Duke Math. J.}, 161(5):735--801, 2012.

\bibitem{DebarreManivel}
O.~Debarre and L.~Manivel.
\newblock Sur la vari\'et\'e des espaces lin\'eaires contenus dans une
  intersection compl\`ete.
\newblock {\em Math. Ann.}, 312(3):549--574, 1998.

\bibitem{SGA45}
P.~Deligne.
\newblock {\em Cohomologie \'etale}.
\newblock Lecture Notes in Mathematics, Vol. 569. Springer-Verlag, Berlin-New
  York, 1977.
\newblock S{\'e}minaire de G{\'e}om{\'e}trie Alg{\'e}brique du Bois-Marie SGA
  $4\frac 1 2$, Avec la collaboration de J. F. Boutot, A. Grothendieck, L.
  Illusie et J. L. Verdier.

\bibitem{Dimca}
A.~Dimca.
\newblock Monodromy and {B}etti numbers of weighted complete intersections.
\newblock {\em Topology}, 24(3):369--374, 1985.

\bibitem{ElmanLam}
R.~Elman and T.~Y. Lam.
\newblock Pfister forms and {$K$}-theory of fields.
\newblock {\em J. Algebra}, 23:181--213, 1972.

\bibitem{Hilbert}
D.~Hilbert.
\newblock Ueber die {D}arstellung definiter {F}ormen als {S}umme von
  {F}ormenquadraten.
\newblock {\em Math. Ann.}, 32(3):342--350, 1888.

\bibitem{Jannsen}
U.~Jannsen.
\newblock Continuous \'etale cohomology.
\newblock {\em Math. Ann.}, 280(2):207--245, 1988.

\bibitem{Jannsenletter}
U.~Jannsen.
\newblock Letter from {J}annsen to {G}ross on higher {A}bel-{J}acobi maps.
\newblock In {\em The arithmetic and geometry of algebraic cycles ({B}anff,
  {AB}, 1998)}, volume 548 of {\em NATO Sci. Ser. C Math. Phys. Sci.}, pages
  261--275. Kluwer Acad. Publ., Dordrecht, 2000.

\bibitem{Kahn}
B.~Kahn.
\newblock Classes de cycles motiviques \'etales.
\newblock {\em Algebra Number Theory}, 6(7):1369--1407, 2012.

\bibitem{KS3}
B.~Kahn and R.~Sujatha.
\newblock Motivic cohomology and unramified cohomology of quadrics.
\newblock {\em J. Eur. Math. Soc. (JEMS)}, 2(2):145--177, 2000.

\bibitem{Kerz}
M.~Kerz.
\newblock The {G}ersten conjecture for {M}ilnor {$K$}-theory.
\newblock {\em Invent. Math.}, 175(1):1--33, 2009.

\bibitem{Krasnov}
V.~A. Krasnov.
\newblock On the equivariant {G}rothendieck cohomology of a real algebraic
  variety and its application.
\newblock {\em Izv. Ross. Akad. Nauk Ser. Mat.}, 58(3):36--52, 1994.

\bibitem{Lam}
T.~Y. Lam.
\newblock {\em The algebraic theory of quadratic forms}.
\newblock Benjamin/Cummings Publishing Co., Inc., Advanced Book Program,
  Reading, Mass., 1980.
\newblock Revised second printing, Mathematics Lecture Note Series.

\bibitem{PfisterStufe}
A.~Pfister.
\newblock Zur {D}arstellung von {$-1$} als {S}umme von {Q}uadraten in einem
  {K}\"orper.
\newblock {\em J. London Math. Soc.}, 40:159--165, 1965.

\bibitem{Pfister17}
A.~Pfister.
\newblock Zur {D}arstellung definiter {F}unktionen als {S}umme von {Q}uadraten.
\newblock {\em Invent. Math.}, 4:229--237, 1967.

\bibitem{PfisterICM}
A.~Pfister.
\newblock Sums of squares in real function fields.
\newblock In {\em Actes du {C}ongr\`es {I}nternational des {M}ath\'ematiciens
  ({N}ice, 1970), {T}ome 1}, pages 297--300. Gauthier-Villars, Paris, 1971.

\bibitem{Pfisterbook}
A.~Pfister.
\newblock {\em Quadratic forms with applications to algebraic geometry and
  topology}, volume 217 of {\em London Mathematical Society Lecture Note
  Series}.
\newblock Cambridge University Press, Cambridge, 1995.

\bibitem{PfiSch}
A.~Pfister and C.~Scheiderer.
\newblock An elementary proof of {H}ilbert's theorem on ternary quartics.
\newblock {\em J. Algebra}, 371:1--25, 2012.

\bibitem{Scheiderer}
C.~Scheiderer.
\newblock {\em Real and \'etale cohomology}, volume 1588 of {\em Lecture Notes
  in Mathematics}.
\newblock Springer-Verlag, Berlin, 1994.

\bibitem{Serreprofini}
J.-P. Serre.
\newblock Sur la dimension cohomologique des groupes profinis.
\newblock {\em Topology}, 3:413--420, 1965.

\bibitem{Voevodsky}
V.~Voevodsky.
\newblock Motivic cohomology with {${\bf Z}/2$}-coefficients.
\newblock {\em Publ. Math. Inst. Hautes \'Etudes Sci.}, (98):59--104, 2003.

\end{thebibliography}

\end{document}